\documentclass[12pt]{article}

\usepackage[margin=2.5cm, top=2.5cm, bottom=2.5cm]{geometry}

\usepackage[utf8]{inputenc} 
\usepackage[T1]{fontenc}    

\usepackage{url}            
\usepackage{booktabs}       
\usepackage{amsfonts}       
\usepackage{nicefrac}       
\usepackage{microtype}      
  
\usepackage{amsmath,amssymb,amsthm,xcolor}
\usepackage{array}
\usepackage{float}
\usepackage{xr-hyper}

\usepackage{wrapfig}
\usepackage{thmtools}
\usepackage[hidelinks,hypertexnames=false]{hyperref}
\usepackage[capitalise]{cleveref}

\theoremstyle{plain}
\newtheorem{theorem}{Theorem}[section]
\declaretheorem[
  name=Fact,          
  sibling=theorem,      
  refname={Fact,Facts}   
]{fact}
\declaretheorem[name=Lemma,       sibling=theorem, refname={Lemma,Lemmas}]{lemma}
\declaretheorem[name=Corollary,   sibling=theorem, refname={Corollary,Corollaries}]{corollary}

\declaretheorem[name=Proposition, sibling=theorem, refname={Proposition,Propositions}]{proposition}
\declaretheorem[name=Assumption,  sibling=theorem, refname={Assumption,Assumptions}]{assumption}
\declaretheorem[name=Example,     sibling=theorem, refname={Example,Examples}]{example}

 \declaretheorem[
  name=Definition,
  sibling=theorem,
  refname={Definition,Definitions},
  style=definition
]{definition}

\usepackage{graphicx}
\usepackage{subcaption}
\usepackage{mdframed,framed}
\usepackage{lipsum}
\usepackage{tikz,calc}
\usepackage{enumitem}
\usetikzlibrary{calc}
\allowdisplaybreaks
\usepackage{pgfplots}
\usetikzlibrary{intersections, pgfplots.fillbetween}
\usepackage{epstopdf}
\usepackage{algorithmic}
\usepackage{booktabs}
\usepackage{changepage}
\usepackage{centernot}
\usepackage{stmaryrd}
\usepackage[nottoc]{tocbibind}
\newcommand{\forae}{%
  \tikz[baseline={(forall.base)}]{
    \node[inner sep=0pt, outer sep=0pt] (forall) {$\forall$};
    \fill[white] (-0.06em,0.05ex) rectangle (0.06em,0.25ex);
    \draw[line width=0.04em] (-0.06em, 0.7ex) -- (0.06em, 0.7ex);
  }
}
\usepackage{empheq}

\newcommand{\R}{\mathbb{R}}
\newcommand{\fg}{\mathfrak{g}}

\newcommand{\cP}{\mathcal{P}}
\newcommand{\dom}{\mathrm{dom}}
\newcommand{\rank}{\mathrm{rank}\hspace{.3mm}}

\newcommand{\gph}{\mathrm{gph}}

\newcommand{\cX}{\mathcal{X}}

\newcommand{\N}{\mathbb{N}}
\newcommand{\eR}{\overline{\R}}
\newcommand{\lip}{\mathrm{lip}}
\def\Im{\mathrm{Im}\hspace{.3mm}}
\def\ker{\mathrm{Ker}\hspace*{.3mm}}
\newcommand{\cl}{\mathrm{cl}}
\def\GL{\mathrm{GL}}
\def\inte{\mathrm{int}}

\def\p{^}
\newcommand{\co}{\mathrm{co}\hspace*{.3mm}}

\usepackage{accents}
\def\cf{\accentset{\circ}{f}}

\def\sign{\mathrm{sign}}
\def\sgn{\mathrm{sgn}}
\def\sp{\mathrm{span}}
\usepackage{accents}
\def\cf{\accentset{\circ}{f}}

\newcommand{\im}{\mathrm{Im}\hspace{.3mm}}
\def\O{\mathrm{O}}

\title{\LARGE Implicit regularization of normalized
gradient descent}

\begin{document}

\author{\large C\'edric Josz\thanks{\url{cj2638@columbia.edu}, IEOR, Columbia University, New York.}}
\date{}

\maketitle

\begin{center}
    \textbf{Abstract}
    \end{center}
    \vspace*{-3mm}
 \begin{adjustwidth}{0.2in}{0.2in}
 ~~~~ How to find flat minima? We propose 
running normalized gradient descent, usually reserved for nonsmooth optimization, with sufficiently slowly diminishing step sizes. This induces implicit regularization towards flat minima if an appropriate Lyapunov functions exists in the gradient dynamics. Our analysis shows that implicit regularization is intrinsically a question of nonsmooth analysis, for which we deploy the full power of variational analysis and stratification theory.
\end{adjustwidth} 
\vspace*{3mm}
\vspace*{3mm}


\tableofcontents

\section{Introduction}
\label{sec:Introduction}

Normalized gradient descent, proposed by Shor in 1962 \cite{Shor1962,shor2012minimization}, seeks to minimize a function $f:\R\p n \to \R$ by choosing a step size schedule $\{\alpha_k\}_{k\in \N}\subseteq(0,\infty)$, an initial point $x_0\in \R\p n$, and iterating
$$\forall k\in \N,~~~ x_{k+1} = x_k - \alpha_k \frac{\nabla f(x_k)}{|\nabla f(x_k)|},$$
assuming $f$ is differentiable at $x_k$ and $\nabla f(x_k)\neq 0$ for all $k \in \N$. For nonsmooth convex functions that attain their infimum, the advantage of normalizing is that it guarantees boundedness of the iterates and
convergence of $\min_{i\in\llbracket 1,k \rrbracket}f(x_i)$
towards the minimum if
$$\sum_{k=0}\p \infty \alpha_k=\infty ~~~\text{and}~~~ \sum_{k=0}\p \infty \alpha_k\p 2<\infty.$$
Notably, coercivity and Lipschitz continuity are not required. In fact, quasiconvexity suffices, as shown by Nesterov \cite{nesterov1984minimization,nesterov2024primal,kiwiel2001convergence}. 

For smooth functions, the advantage is less clear as it provides the suboptimal rate $\min_{i\in \llbracket 0,k\rrbracket} |\nabla f(x_i)| = O(\ln k/\sqrt{k})$, with $\alpha_k = c/\sqrt{k}$ for some $c>0$. 
Empirically, it has been reported that it can help with saddle point evasion and
vanishing or exploding gradients \cite{murray2019revisiting,suzuki2021normalized}. More recently, it was suggested that it can promote implicit regularization \cite{arora2022understanding}.

While there is no formal definition of implicit regularization, explicit regularization means adding a regularizer $g:\R\p n \to \R$ to the objective function $f:\R\p n\to \R$, modulo a parameter $\lambda\geq 0$:
$$\inf_{x\in \R\p n} f(x)+\lambda g(x).$$
When $\lambda >0$, this promotes solutions with lower values of $g$ than if $\lambda=0$. Implicit regularization is the idea that algorithms can do so automatically.

Understanding implicit regularization \cite[Chapter 12]{bach2024learning} is a central challenge in deep learning theory, initiated by Neyshabur \textit{et al.} \cite{neyshabur2015search,neyshabur2015path,neyshabur2017exploring}, and studied in the context of matrix factorization \cite{gunasekar2017implicit,wang2022large}, matrix completion \cite{ma2020implicit}, deep matrix factorization \cite{arora2019implicit}, linear neural networks \cite{gidel2019implicit}, diagonal linear networks \cite{ma2022blessing}, tensor factorization \cite{razin2021implicit}, ReLU neural networks \cite{timor2023implicit,ahn2023learning}, two-layer linear networks \cite{varre2024spectral}, infinitely wide two-layer networks \cite{chizat2020implicit}, and infinitely wide linear networks \cite{chizat2024infinite}. 

Several variants of gradient descent, aside from normalization, have been proposed to promote implicit regularization: large step sizes \cite{nar2018step,wu2018sgd,cohen2021gradient}, sharpness aware minimization \cite{foret2021sharpness,dai2023crucial,bartlett2023dynamics}, unnormalized sharpness aware minimization \cite{andriushchenko2022towards,zhou2024sharpness}, two-step unnormalized sharpness aware minimization \cite{bolte2025convergence}, randomly smoothed perturbed gradient descent \cite{ahn2024escape}, and infinitesimally perturbed-gradient descent \cite{ma2025implicit}.

A lingering question that remains is: 
\begin{center}
    \textit{What quantity should these algorithms be implicitly regularized by?}
\end{center}

Aside from problem specific quantities like certain matrix norms norms or the rank \cite{razin2020implicit}, $\lambda_1(\nabla\p 2 f(x))$ \cite{mulayoff2021implicit,arora2022understanding,marion2024deep} and $|\nabla f(x)|\p 2$ \cite{hairer1997life,barrett2021implicit} have been put forth when $f$ is smooth, where $\lambda_1(\cdot)$ is the top eigenvalue. Our recent work on the Lyapunov stability of the Euler method \cite{josz2025lyapunov}, dubbed d-stability, seeks to address this very question via the notion of $p$-d-Lyapunov function $g:\R\p n\to \R$, where $p\in[1,\infty)$ and `d' stands for discretization. It satisfies 
\begin{equation}
    \label{eq:descent}
    \forall k\in\N,~~~ g(x_{k+1})-g(x_k)\leq -\omega \alpha_k\p p,
\end{equation}
for some $\omega>0$ near a given point, with sufficiently small step sizes. For normalized gradient descent, locally Lipschitz $p$-d-Lyapunov functions $g$ must be Lyapunov functions (i.e., decreasing quantities) of the gradient dynamics
$$\dot{x} = -\nabla f(x)$$
when $f$ is $C\p 1$ semi-algebraic, as we will see. In particular, 
conserved quantities are good candidates for $p$-d-Lyapunov functions. 

In modern applications, the objective function $f$ is typically not coercive, usually due to symmetries, yielding an unbounded set of global minima. Applying the theory of d-stability to normalized gradient descent unlocks the following property when $f$ is locally Lipschitz semi-algebraic. A continuous function $g$ acts as an implicit regularizer towards the global minima of $f+g$ if $f+g$ is coercive, $g$ is $p$-d-Lyapunov near any global minimum of $f$ that is not a global minimum of $g$, and
$$\sum_{k=0}\p \infty \alpha_k\p p=\infty.$$
If additionally $\alpha_k\to 0$, then convergence is also guaranteed. This can be achieved by using
\begin{equation}
\label{eq:step_size}
    \forall k \in \N, ~~~ \alpha_k = \frac{\beta}{(k+1)\p {1/\gamma}},
\end{equation}
with $\gamma\in [p,\infty)$ for some sufficiently small $\beta>0$. Under some mild conditions, the global minima of $f+g$ contain the flat global minima of $f$. Let us illustrate this on an example, which is special case of \cref{eg:monomial}.

\begin{example}
\label{eg:imp}
Normalized gradient descent applied to $f(x,y) = (xy-1)\p 2$ with step size \eqref{eq:step_size} and $\gamma\geq 2$ almost surely converges to its flat minima $\pm(1,1)$.
\end{example}
\begin{proof}
The gradient dynamics of the objective function, namely
    $$\left\{\begin{array}{ccl}
        \dot{x} & = & -2(xy-1)y, \\
        \dot{y} & = & -2(xy-1)x,
    \end{array}\right.$$ admits the conserved quantity $x\p 2-y\p 2$ since 
    $$\frac{d}{dt}(x\p 2-y\p 2) = 2\dot{x}x-2\dot{y}y = -4(xy-1)xy+4(xy-1)xy=0.$$
    Accordingly, let $g:\R\p 2\to \R$ be defined by
    $$g(x,y) = (x\p 2-y\p 2)\p 2.$$
    The global minima of $f+g$ are $\pm (1,1)$ since $f+g\geq 0$ and
    $f(x,y)+g(x,y)=0$ iff $xy=1$ and $x\p 2 = y\p 2$ iff $(x,y)=\pm (1,1)$. They are the only local minima of $\lambda_1(\nabla\p 2f(x,y)) = 2(x\p 2+y\p 2)$ subject to $f(x,y)=(xy-1)\p 2=0$, strict local minima at that. They are hence the flat global minima of $f$. Moreover, the sum $f+g$ is coercive since
    $$\forall (x,y)\in \R\p 2, ~~~f(x,y)+g(x,y)\geq \frac{1}{4}(x\p 4+y\p 4)-1.$$
    Indeed,
    $$x\p 4 + y\p 4 = (x\p 2-y\p 2)\p 2 + 2x\p 2y\p 2 = (x\p 2-y\p 2)\p 2 + 2(xy-1+1)\p 2 \leq g(x,y)+4f(x,y)+4.$$ 
    Finally, for all $(x,y)\in \R\p 2$ with $xy\neq 1$ near a given point $(\widetilde{x},\widetilde{y})$ where $g$ is positive, and for all sufficiently small $\alpha>0$, we have
    \begin{align*}
        g(x\p+,y\p +) & = \left[(x\p +)\p 2 - (y\p +)\p 2\right]\p 2 \\
                      & = \left[\left(x-\alpha\frac{2\hspace{.3mm}\sign(xy-1)y}{\sqrt{y\p 2 + x\p 2}}\right)\p 2-\left(y-\alpha\frac{2\hspace{.3mm}\sign(xy-1) x}{\sqrt{y\p 2 + x\p 2}}\right)\p 2\right]\p 2 \\
                      & = \left[x\p 2+4\alpha\p 2\frac{y\p 2}{y\p 2 + x\p 2}-y\p 2-4\alpha\p 2\frac{x\p 2}{y\p 2 + x\p 2}\right]\p 2 \\
                      & = \left(x\p 2-y\p 2\right)\p 2 \left(1-\frac{4\alpha\p 2}{x\p 2 + y\p 2}\right)\p 2 \\
                      & \leq g(x,y)\left(1 - \frac{8\alpha\p 2}{x\p 2 + y\p 2}\right) \\
                      & \leq g(x,y) - \frac{4g(\widetilde{x},\widetilde{y})}{\widetilde{x}\p 2 + \widetilde{y}\p 2}\alpha\p 2,
    \end{align*}
    implying that $g$ is $2$-d-Lyapunov near $(\widetilde{x},\widetilde{y})$. 
\end{proof}

Normalized gradient descent is displayed in \cref{fig:imp_reg} with step sizes $0.4/(k+1)\p {1/4}$ applied to $f$, while gradient descent is displayed in \cref{fig:exp_reg} with constant step size $0.1$ applied to $f+\lambda g$ with $\lambda = 0.1$. The choice of the exponent $\gamma\in[p,\infty)$ in the step size \eqref{eq:step_size} entails a trade-off: low values favor a decrease of $f$, while high values encourage a decrease of $g$ (although neither are monotone). The value $\gamma=4$ strikes a balance. Taking `$\gamma=\infty$' corresponds to constant step sizes and is a legitimate choice (see \cref{fig:monomial_xy}). On the other hand, $1\leq\gamma<p$ can lead to premature convergence of $g$, and the same can happen with $f$ if $0<\gamma<1$.

\begin{figure}[H]
\centering
\begin{subfigure}{.49\textwidth}
  \centering
  \includegraphics[width=1\textwidth]{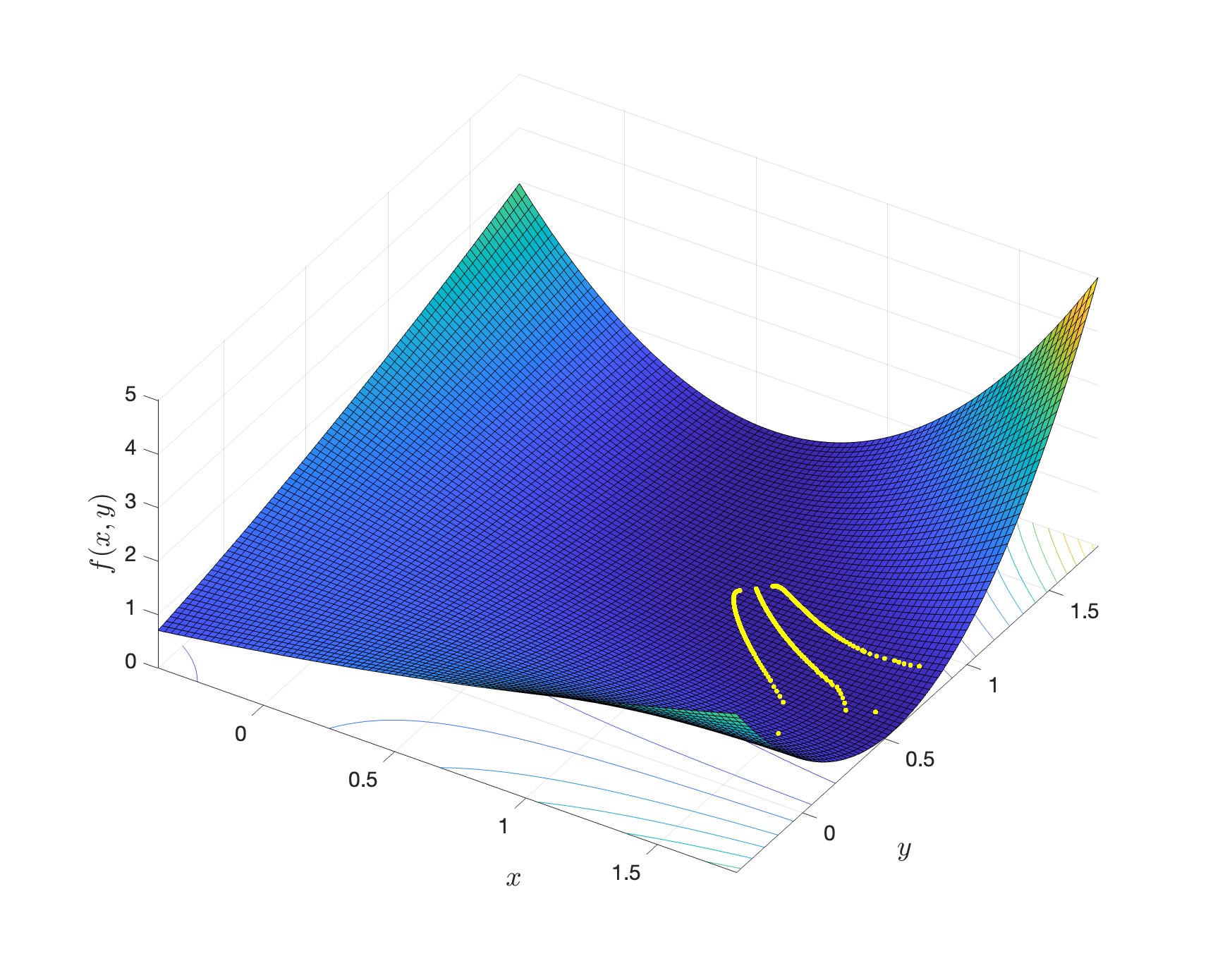}
  \caption{Implicit regularization.}
  \label{fig:imp_reg}
\end{subfigure}
\begin{subfigure}{.49\textwidth}
\centering
  \includegraphics[width=1\textwidth]{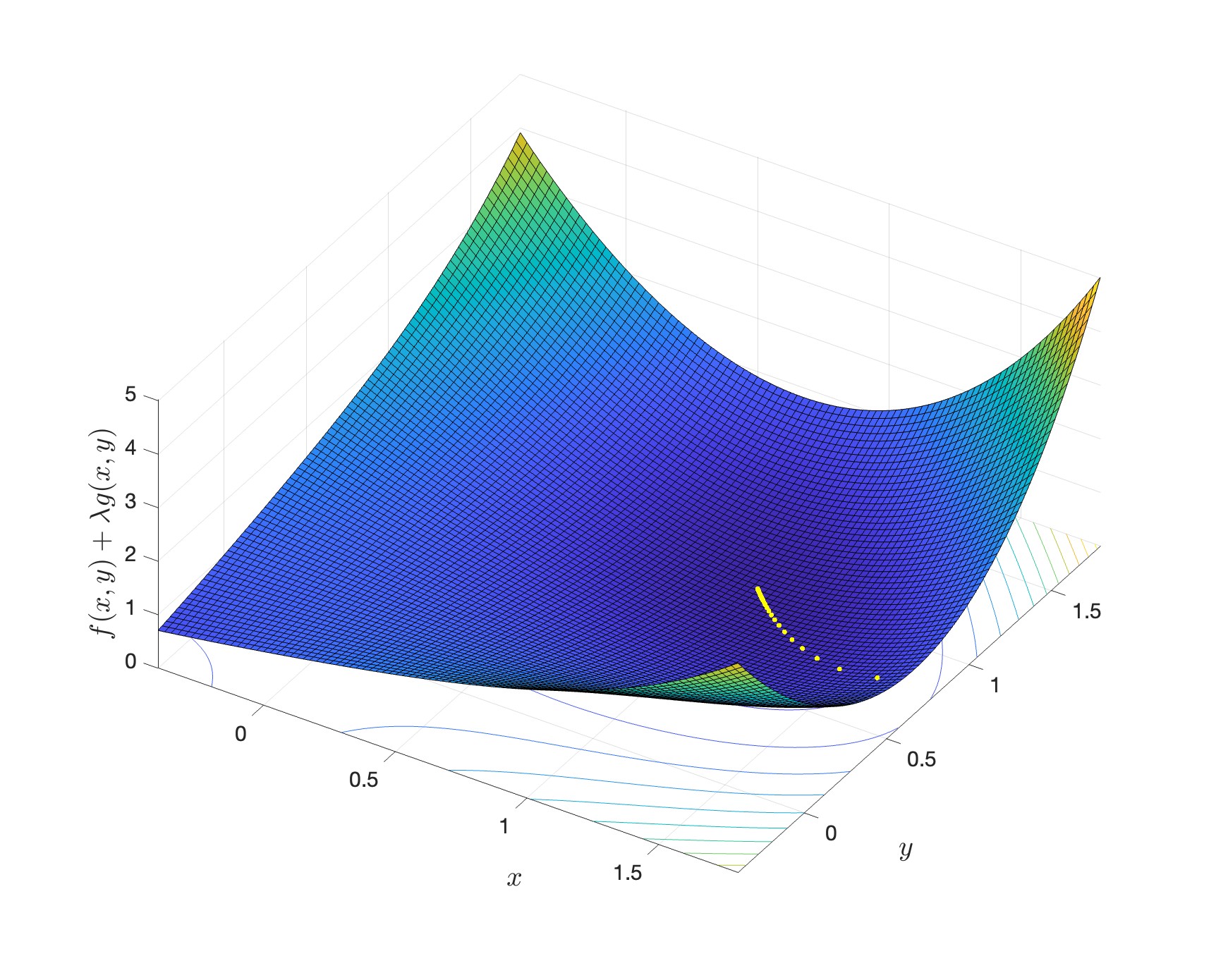}
  \caption{Explicit regularization.}
  \label{fig:exp_reg}
\end{subfigure}
\caption{$f(x,y)=(xy-1)\p 2$, $g(x,y)=(x\p 2-y\p 2)\p 2$, and $\lambda = 0.1$.}
\end{figure}

We propose a series of examples in \cref{sec:Examples}, summarized in \cref{tab:eg}. The function $g$ is a Lyapunov function in all cases, and it is conserved quantity in all but the first example in \cref{tab:eg}.

\begin{table}[H]
\begin{center}
\begin{tabular}{ |c|c|c|c| } 
 \hline
 Example & $f$ & $g$ & $\overline{p}$ \\
 \hline
 \ref{eg:y4} & $y\p 2+ x\p 2 y\p 4$ & $|x|$ & $(0,0)$ \\
 \ref{eg:parabola} & $(x\p 2 - y)\p 2$ & $x\p 2 e\p{4y}$ & $(0,0)$ \\
 \ref{eg:cubic} & $(x\p 3-y\p 3 -1)\p 2$ & $\exp(-\sgn(x+y)(1/x+1/y))$ & $(1,0)$ \\
 \ref{eg:ellipse} & $(2x\p 2 + y\p 2-1)\p 2$ & $|x|/|y|\p 2$ & $\pm(0,1)$ \\
 \ref{eg:monomial} & $(xy\p 2-1)\p 2$ & $(2x\p 2 - y\p 2)\p 2$ & $(2^{-1/3},\pm 2^{1/6})$ \\ 
 \ref{eg:mf1_01} & $|xz|+|yz-1|$ & $(x\p 2 + y\p 2 - z\p 2)\p 2$ & $\pm (0,2^{1/4},2^{-1/4})$ \\ 
 \hline
\end{tabular}
\caption{Objective function $f$, implicit regularizer $g$, and flat minimum $\overline{p}$.}
\label{tab:eg}
\end{center}
\end{table}

We next summarize the contributions of this manuscript:
\begin{enumerate}
    \item We define the normalized subdifferential as a set-valued mapping $\widehat\nabla f:\R\p n\rightrightarrows\R\p n$, enabling a rigorous definition of normalized gradient descent for any objective function $f:\R\p n\to\R$, in particular when $\nabla f(x_k)=0$ or $f$ is not differentiable at $x_k$. This translates the study of normalized gradient descent for smooth functions, which could be viewed as gradient descent with adapative step sizes, into a question of nonsmooth analysis.
    \item We study basic variational properties of $\widehat\nabla f$. For locally Lipschitz functions, it is upper semicontinuous with nonempty compact values, a crucial property for analyzing stability. What makes the analysis particularly interesting is that it is not a conservative field for $f$, contrary to commonly encountered set-valued mappings like the Bouligand subdifferential $\overline{\nabla} f$ or the Clarke subdifferential $\overline{\partial }f$.
    \item We move on to stratification properties of $\widehat\nabla f$. It helps us to understand the effect of discretizing normalized gradient dynamics on stability for locally Lipschitz semi-algebraic functions. While stability and local optimality are equivalent in continuous time, the latter becomes a necessary condition in discrete time. Importantly, it is no longer sufficient, paving the way to implicit regularization.
    \item We propose necessary and sufficient conditions for checking whether a given function can serve as an implicit regularizer. 
    For sufficiency, we rely on a normalized projection formula and metric subregularity. We show that the former holds if the objective function admits a certain composite structure or it is invariant under the action of a Lie group. We also show how to check the decrease condition \eqref{eq:descent} without resorting to a sequence, simply by computing derivatives of $f$ and $g$ at a global minimum of $f$.
    \item We forge a link between stability and flatness. 
    It draws on the analysis of the gradient dynamics of the implicit regularizer, carried out in our recent work on the geometry of flat minima \cite{josz2025flat}. More generally, we explore the relationship between stability, flatness, conservation, and symmetry. In particular, we show how to build an implicit regularizer using conserved quantities derived from linear symmetries, coming full circle. 
\end{enumerate}

This paper is organized as follows. \cref{sec:background} contains background material on set-valued analysis, implicit regularization, and Lyapunov stability. \cref{sec:Theory} develops the theory that underpins implicit regularization of normalized gradient descent. \cref{sec:Examples} illustrates the theory on various numerical examples. Finally, \cref{sec:Appendix} contains the Appendix.

\section{Background}
\label{sec:background}

We first introduce some notations and vocabulary. As usual, $\N=\{0,1,2,\hdots\}$, $\N\p * = \N\setminus \{0\}$, $\R\p * = \R\setminus\{0\}$, $\R_+=[0,\infty)$, and $\eR=\R\cup\{\infty\}$. Given two integers $a \leq b$, let $\llbracket a,b \rrbracket = \{a,a+1,\hdots,b\}$. The symbols $\land$ and $\forae$ respectively mean `and' and `for almost every'. The indicator and characteristic functions of a subset $A\subseteq \R\p n$ are denoted by $\delta_A:\R\p n \to\eR$ and $\chi_A:\R\p n \to \{0,1\}$ respectively. 

Let $|\cdot|=\sqrt{\langle\cdot,\cdot\rangle}$ denote the Euclidean norm in $\R\p n$, and $B_r(x)$ denote the corresponding open ball of radius $r$ centered at $x\in\R\p n$. Let $S\p {n-1}$ denote the unit sphere in $\R\p n$. Let $\|\cdot\|_F=\sqrt{\langle\cdot,\cdot\rangle_F}$ and $\|\cdot\|_1$ respectively denote the Frobenius norm and the entrywise $\ell_1$-norm of a matrix. Let $\lambda(\cdot)$ and $\sigma(\cdot)$ respectively denote the vector of eigenvalues and singular values of a matrix in decreasing order of magnitude, and let $\rho(\cdot)$ denote the spectral radius.

A function $\varphi:A\to\eR$ where $A\subseteq \R$ is increasing (resp. strictly increasing) if $\varphi(s)\leq \varphi(t)$ (resp. $\varphi(s)<\varphi(t)$) for all $s,t\in A$ such that $s<t$. It is decreasing (resp. strictly decreasing) if $-\varphi$ is increasing (resp. strictly increasing).
Let $D\p k$ (resp. $C\p k$) denote the set of functions from $\R\p n$ to $\R$ that are $k$-times differentiable (resp. continuously differentiable). Let $C\p {k,\ell}$ denote the set of $C\p k$ functions whose $\ell$\textsuperscript{th} derivative is locally Lipschitz continuous.

Let $\overline{(\cdot)}$, $\mathrm{int}$, $\mathrm{relint}$, and $\mathrm{co}$ respectively denote the closure, interior, relative interior, and convex hull. Given a point $x \in \mathbb{R}^n$, we use 
\begin{equation*}
    d(x,X) = \inf \{|x-y|: y \in X\}~~~\text{and}~~~
    P_X(x) = \arg\min \{ |x-y| : y \in X\}
\end{equation*}
to denote the distance to and the projection on a subset $X\subseteq \mathbb{R}^n$, respectively. The following notations will come in handy:
$$|X|_- = \inf\{|x|:x\in X\}~~~\text{and}~~~|X|_+ = \sup\{|x|:x\in X\}.$$
The distance \cite[(2.1) p. 197]{kato2013perturbation} between two linear subspaces $V,W$ of $\mathbb{R}^n$ is given by 
\begin{equation*}
    d(V,W) = \sup \{ d(v,W) : v \in V, |v|=1\}
\end{equation*}
and $d(V,W) = 0$ if $V = \{0\}$. We use $V\p \perp$ to denote the orthogonal subspace to $V$.

A subset $A$ of $\mathbb{R}^n$ is semi-algebraic \cite{bochnak2013real} if it is a finite union of basic semi-algebraic sets, which are of the form 
    \begin{equation*}
        \{ x \in \mathbb{R}^n : f_1(x) > 0 , \hdots , f_p(x) > 0, f_{p+1}(x) = 0, \hdots , f_q(x) = 0\}
    \end{equation*}
    where $f_1,\hdots,f_q$ are polynomials with real coefficients. A function $f:\mathbb{R}^n\rightarrow\overline{\mathbb{R}}$ is semi-algebraic if $\gph f$ is semi-algebraic. All the results in this manuscript hold more generally in an o-minimal structure on the real field \cite{van1998tame}.

Let $T_p M$ and $N_p M$ respectively denote the tangent and normal spaces at $p\in M$ of a submanifold $M$ of $\R\p n$ \cite{lee2012smooth}. $TM$ and $NM$ denote the corresponding bundles. The differential of a smooth map $F:M\to N$ at $p\in M$ is denoted $dF_p:T_pM\to T_{F(p)}N$. The sets $\Im A$ and $\ker A$ respectively denote the image and Kernel of a linear map $A:V\to W$. The adjoint of a linear map $A:V\to W$ between inner product spaces $V,W$ is denoted $A\p *$.

\subsection{Set-valued analysis}
 
Given $f:\R^n\to\eR$ and $\ell \in {\mathbb{R}}$, let 
$$[f = \ell] = \{ x \in \mathbb{R}^n : f(x) = \ell \}$$
and define $[f \leq \ell]$ similarly. In particular, $\arg\min f = [f=\min f]$. The domain and the graph of $f$ are defined by 
$$\dom f = \{ x \in \R\p n : f(x)<\infty\} ~~~\text{and}~~~\gph f = \{ (x,\ell) \in \R\p {n+1} : f(x)=\ell\}.$$
For a set-valued mapping $F:\mathbb{R}^n \rightrightarrows \mathbb{R}^m$, the image of $X\subseteq \R\p n$ and the preimage of $y \in \R\p m$ are given by 
$$F(X) = \bigcup_{x \in X} F(x) ~~~\text{and}~~~ F^{-1}(y) = \{ x\in \mathbb{R}^n : F(x) \ni y \}.$$
A mapping $F:\R\p n \rightrightarrows \R\p m$ is locally bounded \cite[Definition 5.14]{rockafellar2009variational} at $\overline{x}\in \R\p n$ if there exists a neighborhood $U$ of $\overline{x}$ such that $F(U)$ is bounded. It is upper semicontinuous \cite[Definition 1 p. 41]{aubin1984differential} at $\overline{x} \in \R\p n$ if for any neighborhood $V$ of $F(\overline{x})$, there exists a neighborhood $U$ of $\overline{x}$ such that $F(U)\subseteq V$.
    It is outer semicontinuous at $\overline{x} \in \mathbb{R}^n$ if $\limsup_{x\rightarrow \overline{x}} F(x) \subseteq F(\overline{x})$ and inner semicontinuous at $\overline{x} \in \mathbb{R}^n$ if $\liminf_{x\rightarrow \overline{x}} F(x) \supseteq F(\overline{x})$ where
    \begin{align*}
        \limsup_{x\rightarrow \overline{x}} F(x) & = ~ \bigcup_{x_k \rightarrow \overline{x}} \limsup_{k\rightarrow \infty} F(x_k) \\
        & = ~ \{ y \in \mathbb{R}^m : \exists (x_k,y_k) \in \gph F \rightarrow (\overline{x},y) \}, \\[2mm]
        \liminf_{x\rightarrow \overline{x}} F(x) & = ~ \bigcap_{x_k \rightarrow \overline{x}} \liminf_{k\rightarrow \infty} F(x_k) \\
        & = ~ \{ y \in \mathbb{R}^m : \forall x_k \rightarrow \overline{x}, \exists y_k \rightarrow y:  (x_k,y_k) \in \gph F ~ \text{eventually} \}.
    \end{align*}
    It is continuous at $\overline{x}$ if it is both outer and inner semicontinuous at $\overline{x}$.
The notions hold without referring to a point if they hold at every point in $\R\p n$.

A curve is a function $x:I\to \R\p n$ where $I$ is an interval of $\R$. A trajectory of $F:\R\p n \rightrightarrows\R\p n$ is an absolutely continuous curve $x:I\to\R\p n$ such that
$$\forae t\in I, ~~~ x'(t) \in F(x(t)).$$
We refer to solutions to the differential inclusion $\dot{x}\in F(x)$ as trajectories of $F$ over an interval $I = [0,T)$ for some $T\in(0,\infty]$ or $I = [0,T]$ for some $T\in(0,\infty)$. A solution $x:I\to\R\p n$ is maximal if for any other solution $y:J\to\R\p n$ such that $I\subseteq J$ and $x(t)=y(t)$ for all $t\in I$, we have $I=J$. A solution is globally defined if $I = [0,\infty)$. If $F:\R\p n \rightrightarrows\R\p n$ is upper semicontinuous with nonempty compact convex values, then the initial value problem 
$$\left\{
\begin{array}{cc}
     \dot{x} \in F(x),  \\
     x(0) = x_0,
\end{array}
\right.$$
admits a solution for any initial condition $x_0\in \R\p n$ by \cite[Theorem 3 p. 98]{aubin1984differential}. As a result, bounded maximal solutions are globally defined.

Recall that, given $F:\mathbb{R}^n \rightrightarrows \mathbb{R}^n$, a point $\overline{x} \in \mathbb{R}^n$ is stable if
$$ \forall \epsilon >0, ~ \exists \delta>0: ~~~
        x(0) \in B_{\delta}(\overline{x}) ~~~ \Longrightarrow ~~~ 
        x([0,T)) \subseteq B_{\epsilon}(\overline{x}) $$
where $x:[0,T)\rightarrow \mathbb{R}^n$ is any trajectory of $F$. 
It is asymptotically stable \cite[Definition 2.1]{clarke1998asymptotic} if it is stable and there exists $\delta_0>0$ such that all maximal trajectories initialized in $B_{\delta_0}(\overline{x})$ are globally defined and converge to $\overline x$. We say that $f:\R\p n\to\eR$ is a Lyapunov function for $F$ on $U\subseteq \R\p n$ if $f\circ x$ is decreasing for any trajectory of $F$ with values in $U$.

Stability can be adapted to difference inclusions: given $G:\mathbb{R}^n \rightrightarrows \mathbb{R}^n$, a point $\overline{x} \in \mathbb{R}^n$ is stable if
$$ \forall \epsilon >0, ~ \exists \delta>0: ~~~
        x_0 \in B_{\delta}(\overline{x}) ~~~ \Longrightarrow ~~~ 
        \{x_k\}_{k\in\N} \subseteq B_{\epsilon}(\overline{x}) $$
for any $\{x_k\}_{k\in \N} \subseteq \mathbb{R}^n$ such that  $x_{k+1} \in G(x_k)$ for all $k\in\N$. It is asymptotically stable if it is stable and there exists $\delta_0>0$ such that if $x_0 \in B_{\delta_0}(\overline{x})$, then $x_k\to\overline x$. By analogy to difference inclusions, we will refer to difference inclusions as $x\p + \in G(x)$. The above definitions of stability can be weakened to an almost sure sense if the initial point lies outside a null set.

\subsection{Implicit regularization}

As stated in the introduction, implicit regularization is the idea that algorithms behave as though one were using an explicit regularizer. There are two main avenues to achieving this: either by carefully choosing the initial point, or by preventing convergence to certain solutions. 

We illustrate the first avenue with two applications. The first application concerns linear least-squares and is inspired by Bach \cite[Section 12.1.1]{bach2024learning}. Let $A \in \R\p {m \times n}\setminus\{0\}$ and $b\in \R\p m$. Consider 
$f,g:\R\p n \to \R$ defined by
$$f(x) = |Ax-b|\p 2 ~~~\text{and}~~~g(x)=|x|\p 2.$$
If gradient descent with constant step size $\alpha \in (0,1/(2\sigma_1(A)\p 2))$ is initialized at $x_0 = A\p T \lambda$ for some $\lambda \in \R\p m$, then 
$$\lim x_k \in \arg\min \{ g(x) : f(x)=\min f \}.$$
This follows from the more general fact $\lim x_k = P_{\arg\min f}(0) + P_{\ker A}(x_0)$. The proof is immediate after reducing to the diagonal case via a singular value decomposition of $A$ and \cref{prop:change}.

The second application, matrix factorization, dates back to Baldi and Hornik \cite{baldi1989neural}. We consider the overparmatrized setting: given $M \in \R\p{m\times n}$ and $r \geq \rank(M)$, consider $f,g:\R\p{m\times r}\times \R\p{r\times n}\to\R$ defined by
$$f(X,Y) = \|XY-M\|_F\p 2 ~~~\text{and} ~~~ g(X,Y) = \|X\p TX-YY\p T\|_F\p 2.$$
The objective function $f$ admits a conserved quantity $C:\R\p{m\times r}\times \R\p{r\times n}\to\R\p{r\times r}$ defined by $C(X,Y) = X\p TX-YY\p T$.
Let $S\subseteq \R\p{m\times r}\times \R\p{r\times n}$ be a bounded open set containing $(0,0)$. There exists $\overline{\alpha}>0$ such that for all step sizes $\alpha \in (0,\overline{\alpha}]$ and almost every initial point $(X_0,Y_0)\subseteq S$,
gradient descent converges to a global minimum $(X_\infty\p \alpha,Y_\infty\p \alpha)$ by \cite[Example 1]{josz2023global}. 
If $S \ni (X_0,Y_0)\to(0,0)$ and $\alpha \to 0$, then $(X_\infty\p \alpha,Y_\infty\p \alpha)\to (X_\infty,Y_\infty)$ up to a subsequence where $C(X_\infty,Y_\infty) = C(0,0)= 0$. 
Thus
$$(X_\infty,Y_\infty) \in \arg\min \{ g(X,Y) : f(X,Y) = 0 \}.$$
In particular, $(X_\infty,Y_\infty)$ is a flat global minimum of $f$ by \cite[Lemma 4.19, Corollary 4.24, Proposition 4.26]{josz2025flat}.

The second avenue was initiated by Nar and Sastry \cite[Proof of Theorem 1]{nar2018step} and is based on the center stable manifold theorem \cite{shub2013global}.
Suppose $G:\R\p n \to \R\p n$ is $C\p 1$ near a fixed point $\overline{x}\in \R\p n$. The dynamical system $x\p + = G(x)$ satisfies:
\begin{itemize}
    \item If $\rho(DG(\overline{x}))<1$, then $\overline{x}$ is asymptotically stable.
    \item If $\overline{x}$ is almost surely stable, then $\rho(DG(\overline{x}))\leq 1$.
\end{itemize}
This applies to the Euler discretization of $\dot{x}=F(x)$ with constant step size $\alpha>0$ by letting $G= \mathrm{Id}_{\R\p n}+\alpha F$, when $F:\R\p n\to \R\p n$ is $C\p 1$. If $F$ is semi-algebraic, one say more: for all but finitely many $\alpha$, and for almost every $x_0\in\R\p n$, if 
$x_k \to \overline{x}$, then $\rho(DG(\overline{x}))\leq 1$ \cite[Theorem 2.1]{bolte2025convergence} (see also \cite[Proposition 2.5]{cheridito2024gradient} and \cite[Footnote 7, Proposition 4.3]{ouyang2025kurdyka}). 

In particular, when $F = -\nabla f$ for a $C\p 2$ function $f:\R\p n \to \R$, then $DG = \mathrm{Id}_{\R\p n} -\alpha\nabla\p 2 f$ is symmetric and $\lambda(DG) = 1-\alpha \lambda(\nabla \p 2 f)$. Hence, when $f$ is semi-algebraic, for almost every $(\alpha,x_0)\in (0,\infty)\times \R\p n$, if gradient descent converges to a point $\overline{x}$, then $\lambda_i(\nabla\p 2 f(\overline{x}))\subseteq [0,2/\alpha]$ for all $i\in\llbracket 1,n\rrbracket$. This suggests that gradient descent with large step size is implicitly regularized by $\lambda_1(\nabla f)$. 

Proving convergence with large step size is challenging however: even \cref{eg:imp} is currently out of reach. Converging to a flat minimum almost surely implies a single choice of step size $\alpha = 1/2$ since $\{\lambda_1(\nabla\p 2 f(x,y)): xy=1\}=[4,\infty)$. Numerically, this leads to intricate nonconverging behavior around the flat minima. As for the aforementioned variants of gradient descent, their analysis is much more convoluted. It is also not clear if it enables finding flat minima in \cref{eg:imp}. 


\subsection{Lyapunov stability}
\label{subsec:Lyapunov_stability}

Instead of relying on the center stable manifold theorem, we propose to use a recent extension \cite{josz2025lyapunov} of Lyapunov's stability criterion for analyzing the stability of the Euler method. There are several reasons for doing so. Either ones views normalized gradient descent as a nonsmooth dynamical system with a predefined step size schedule, or one views it as a smooth dynamical system (if the objective is smooth) with adaptive step sizes. In either case, the center stable manifold theorem does not apply (even in \cite{mucsat2025gradient} the adaptive step sizes are eventually constant). 

Regardless, the theorem only yields local convergence. When initializing near a point where it prevents convergence almost surely, nothing is known about the asymptotic behavior unless convergence is assumed. Assuming convergence is rather strong assumption which also does not allow for mere stability, a setting of practical interest. 

We next spell out the main ingredients in \cite{josz2025lyapunov}. 
As explained in \cite{josz2025lyapunov}, in order to study the stability of Euler discretizations of discontinuous dynamics, it is useful to introduce the notion of d-stability.

\begin{definition}
\label{def:stable_set}
Given $F:\mathbb{R}^n \rightrightarrows \mathbb{R}^n$, a set $X \subseteq \mathbb{R}^n$ is d-stable if
$$ \forall \epsilon >0, ~ \exists \delta,\overline{\alpha}>0: ~ \forall \{\alpha_k\}_{k \in \N} \subseteq (0,\overline{\alpha}], ~~~ 
        x_0 \in B_{\delta}(X) ~~~ \Longrightarrow ~~~ 
        \{x_k\}_{k\in \N} \subseteq B_{\epsilon}(X) $$
for any $\{x_k\}_{k\in \N} \subseteq \mathbb{R}^n$ such that  $x_{k+1} \in x_k + \alpha_k F(x_k)$ for all $k\in\N$.
\end{definition}

It is natural to also define asymptotic d-stability.

\begin{definition}
\label{def:asym_stable_set}
Given $F:\mathbb{R}^n \rightrightarrows \mathbb{R}^n$, $(p,q)\in[1,\infty)\times \N\p *$, a set $X \subseteq \mathbb{R}^n$ is asymptotically $(p,q)$-d-stable if it is d-stable and
$$ \exists \delta_0,\overline{\alpha}_0>0: ~ \forall \{\alpha_k\}_{k \in \N}\subseteq (0,\overline{\alpha}_0], ~~~ 
        x_0 \in B_{\delta_0}(X) ~~~ \Longrightarrow ~~~ d(x_k,X) \to 0 $$
for any $\{x_k\}_{k\in \N} \subseteq \mathbb{R}^n$ such that $x_{k+1} \in x_k + \alpha_k F(x_k)$ for all $k\in\N$ where 
$\alpha_k\to 0$ and $\sum_{k=0}\p \infty \alpha_{kq+i}\p  p=\infty$ for all $i\in\llbracket 0,q-1\rrbracket$.
\end{definition}

The requirement on the step size is satisfied by $\alpha_k = \beta/(k+1)\p{1/\gamma}$ with $\beta>0$ and $\gamma\geq p$. 
The next definition is a kind of semiglobal d-stability.
\begin{definition}
\label{def:attractor}
    Given $F:\mathbb{R}^n \rightrightarrows \mathbb{R}^n$, $f:\mathbb{R}^n \rightarrow \R$, and $(p,q)\in [1,\infty)\times\N\p *$, a bounded set $A \subseteq \arg\min f$ is a $p$-attractor if for any $\epsilon,\ell>0$ and any bounded set $X \subseteq \arg\min f$, 
$$
    \exists \delta,\overline{\alpha}>0:~ \forall \{\alpha_k\}_{k\in\N} \subseteq (0,\overline{\alpha}],~ \forall x_0 \in B_{\delta}(X), ~~~ f(x_k)\leq \ell ~~\land~~ \exists k_0 \in \mathbb{N}: \{x_k\}_{k\geq k_0} \subseteq B_{\epsilon}(A)
$$
for any $\{x_k\}_{k\in \mathbb{N}}\subseteq \R\p n$ such that $x_{k+1} \in x_k + \alpha_k F(x_k)$ for all $k \in \mathbb{N}$ where $\sum_{k=0}\p \infty \alpha_{kq+i}\p  p=\infty$ for all $i\in\llbracket 0,q-1\rrbracket$. It is an asymptotic $(p,q)$-attractor if in addition it is asymptotically $(p,q)$-d-stable. 
\end{definition}

In order to fulfill these definitions, we define d-Lyapunov functions.

\begin{definition}
    \label{def:pdL}
    Given $F:\R\p n \rightrightarrows \R\p n$, $g:\R^n\to\eR$ is d-Lyapunov on $U \subseteq \R^n$ if
    $$\exists \overline{\alpha}>0: ~ \forall \alpha\in (0,\overline{\alpha}], ~ \forall x \in U,~ \forall u\in F(x), ~~ g(x+\alpha u) \leq g(x).$$
\end{definition}


It is often enviable to either have a stronger decrease or a multistep decrease.

\begin{definition}
    \label{def:pqdL}
    Given $F:\R\p n \rightrightarrows \R\p n$, $g:\R^n\to\eR$ is $(p,q)$-d-Lyapunov on $U \subseteq \R^n$ with $(p,q)\in[1,\infty)\times\N\p *$ if there exist $\overline{\alpha},\omega>0$ such that 
\begin{equation*}
        \forall \{\alpha_k\}_{k\in\N} \subseteq (0,\overline{\alpha}], ~ \forall k \in \N,~~~ x_k \in U ~~~ \Longrightarrow ~~~ g(x_{k+q}) - g(x_k) \leq - \omega \min\{\alpha_k,\hdots,\alpha_{k+q-1}\}\p p
\end{equation*}
for any $\{x_k\}_{k\in\N}\subseteq \R\p n$ such that $x_{k+1}\in x_k +\alpha_k F(x_k)$ for all $k\in\N$.
\end{definition}

When $q=1$, we drop the index and simply speak of $p$-d-Lyapunov, for which we have the descent property $g(x_{k+1})-g(x_k) \leq - \omega \alpha\p p$, ensuring that it is a d-Lyapunov function. A sufficient condition is as follows.

\begin{proposition}
    \label{prop:sufficient_dL2}
    Let $F:\R\p n \rightrightarrows \R\p n$ be locally bounded with a closed graph and $g:\R\p n \to \overline{\R}$ be $C\p {2,2}$ near $\overline{x}\in \R\p n$. If 
    $$\exists r>0:~ \forall x \in B_r(\overline{x}),~ \forall u\in F(x),~ \langle \nabla g(x) , u \rangle \leq 0 ~~~\land ~~~ \sup_{u\in F(\overline{x})} \langle \nabla\p2 g(\overline{x})  u , u \rangle<0,$$
    then $g$ is 2-d-Lyapunov near $\overline{x}$.
    \end{proposition}

The study of d-stability rests on some key assumptions, which call for the following notion.
Given a locally Lipschitz function $f:\R^n\to \R$, a set-valued map $D:\R^n\rightrightarrows\R^n$ is a conservative field \cite[Lemma 2]{bolte2020conservative} for $f$, called potential, if $D$ has a closed graph with nonempty compact values and 
$$  \forae t\in (0,1), ~ \forall v \in D(x(t)),~~~ (f\circ x)'(t)=\langle v, x'(t) \rangle $$
for any absolutely continuous curve $x:[0,1]\to\R^n$. 
Let $\cP$ be the set of continuous functions $\kappa:\R\p n\to \R$ that are positive definite, i.e., such that $\kappa(0)=0$ and $\kappa(x)>0$ for all $x\in \R\p n\setminus\{0\}$.

\begin{assumption}
    \label{assume:Ff}
    Let
    \begin{enumerate}[label=\rm{(\rm{\roman*})}]
        \item $F:\R^n\rightrightarrows \R^n$ be upper semicontinuous with nonempty compact values; \label{item:F} 
        \item $f:\R^n\to\R$ be a potential of a conservative field $D:\R^n\rightrightarrows\R^n$ such that 
        $$\forall \widehat{x}\in \R\p n,~ \exists (\rho,\kappa)\in (0,\infty)\times \cP,~ \forall x\in B_{\rho}(\widehat{x}),~\max_{u\in  \co F(x)} \min_{v\in D(x)} \langle u,v\rangle \leq -\min_{w\in D(x)}\kappa(w)$$
        and $(0,\overline{\ell})\cap f(D\p {-1}(0)) = \emptyset$ for some $\overline{\ell}>0$. \label{item:f} 
\end{enumerate}
\end{assumption}

We can now state the main results on d-stability. The first deals with points.

\begin{theorem}
\label{thm:stable_point}
Under \cref{assume:Ff}, let $\overline{x}\in[f=0]$ and suppose
    \begin{enumerate}[label=\rm{(\rm{\roman*})}]
        \item $g:\R^n\to \eR$ is continuous near $\overline x$;
        \item $g$ is d-Lyapunov near $\overline x$;
        \item $\overline x$ is a strict local minimum of $f+g$.
    \end{enumerate}
Then $\overline x$ is d-stable.
\end{theorem}

The second extends to sets and asymptotic d-stability.

\begin{theorem}
\label{thm:stable_set}
Under \cref{assume:Ff} with $\min f = 0$, suppose
\begin{enumerate}[label=\rm{(\rm{\roman*})}]
    \item $g:\R^n\to\eR$ is continuous near $\arg\min f$ with $\min g = 0$;
    \item $g$ is d-Lyapunov (resp. $p$-d-Lyapunov) near every point in $\arg\min f\cap [g>0]$ near $\arg\min f+g$;
    \item $\arg\min f+g$ is bounded; 
\end{enumerate}
Then $\arg\min f+g$ is d-stable (resp. asymptotically $p$-d-stable). 
\end{theorem}

The third deals with attractors. It captures the idea of implicit regularization, where $g$ acts as an implicit regularizer.

\begin{theorem}
\label{thm:attractor}
Under \cref{assume:Ff} with $\min f = 0$, suppose
\begin{enumerate}[label=\rm{(\rm{\roman*})}]
    \item $g:\R^n\to\eR$ is continuous on its full measure domain with $\min g = 0$;
    \item $g$ is $p$-d-Lyapunov near every point in $\arg\min f\cap [0<g<\infty]$;
    \item $f+g$ is coercive.
\end{enumerate}
Then $\arg\min f+g$ is an almost sure asymptotic $p$-attractor, which is an asymptotic $p$-attractor if $\dom g=\R\p n$.
\end{theorem}
In order to apply the theory of d-stability to normalized gradient descent, various tasks lay ahead. What set-valued mapping $F:\R\p n\rightrightarrows \R\p n$ should we use? Does it satisfy \cref{assume:Ff}? What properties do we require from the objective function $f$? Is smoothness required? Where to look for candidate d-stable points? Should they be local minima of $f$? Flat local minima of $f$? What about d-Lyapunov functions? How to check whether a given function is d-Lyapunov? How to build one? Can we use symmetries and conservation laws? 

\section{Theory}
\label{sec:Theory}

We first recall some basic notions from variational analysis. Given $f:\mathbb{R}^n\rightarrow \overline{\mathbb{R}}$ and a point $\overline{x}\in\mathbb{R}^n$ where $f(\overline{x})$ is finite, the regular subdifferential, subdifferential, horizon subdifferential \cite[Definition 8.3]{rockafellar2009variational}, and Clarke subdifferential of $f$ at $\overline{x}$ \cite[Definition 4.1]{drusvyatskiy2015curves} are respectively given by
\begin{gather*}
    \widehat{\partial} f (\overline{x}) = \{ v \in \mathbb{R}^n : f(x) \geq   f(\overline{x}) + \langle v , x - \overline{x} \rangle + o(|x-\overline{x}|) ~\text{near}~ \overline{x} \}, \\
    \partial f(\overline{x}) = \{ v \in \mathbb{R}^n : \exists (x_k,v_k)\in \gph\hspace*{.5mm}\widehat{\partial} f: (x_k, f(x_k), v_k)\rightarrow(\overline{x}, f(\overline{x}), v) \}, \\[1mm]
    \partial^\infty f(\overline{x}) = \{ v \in \mathbb{R}^n : \exists (x_k,v_k)\in \gph\hspace*{.5mm}\widehat{\partial} f: \exists \tau_k \searrow 0: (x_k, f(x_k), \tau_kv_k)\rightarrow(\overline{x}, f(\overline{x}), v) \}, \\[2mm]
    \overline{\partial} f(\overline{x}) = \overline{\mathrm{co}} [\partial f(\overline{x}) + \partial^\infty f(\overline{x})].
\end{gather*}

Given $f:\R\p n\to \R$, the Bouligand subdifferential at $\overline{x}\in \R\p n$ is defined by
$$\overline{\nabla} f(\overline{x}) = \left\{ v \in \mathbb{R}^n : \exists x_k \xrightarrow[\Omega]{} \overline{x} ~\text{with}~ \nabla f(x_k) \rightarrow v\right\}$$
where $\Omega$ are the differentiable points of $f$. The letter $\Omega$ under the arrow means $x_k \in \Omega$ and $x_k \rightarrow \overline{x}$. Given $\Phi:\R\p n\rightrightarrows\R\p n$, a point $x\in \R^n$ is $\Phi$-critical if $0\in \Phi(x)$ and a scalar $\ell$ is a $\Phi$-critical value of $f$ if there exists $x\in \R^n$ such that $\ell=f(x)$ and $0\in \Phi(x)$.

The Lipschitz modulus of a function $f:\R^n\to\R$ at a point $\overline{x}\in \R\p n$ is defined by \cite[p. 354]{rockafellar2009variational}
$$
    \lip f(\overline{x}) = \limsup_{\scriptsize \begin{array}{c}x,y \rightarrow \overline{x}\\ x\neq y\end{array}} \frac{|f(x)-f(y)|}{|x-y|}.
$$
The function $f$ is Lipschitz continuous near $\overline{x}$ iff $\lip f(\overline{x})<\infty$, in which case $\partial f(\overline{x})$ is nonempty and compact, and $\lip f(\overline{x}) = \max \{|v| : v \in \partial f(\overline{x})\}$ by \cite[Theorem 9.13]{rockafellar2009variational}.

\subsection{Normalized subdifferential}
\label{subsec:Normalized subdifferential}


The Bouligand subdifferential is well-suited for analyzing d-stability with sharp nonsmooth functions, as shown in \cite{josz2025lyapunov}, but with smooth functions another subdifferential is called for. Indeed, the Bouligand subdifferential is simply the gradient in the smooth case, and gradient dynamics with sufficiently nonsummable small step sizes can converge to any local minimum \cite{josz2025reachability}. The following definition is new to the best of our knowledge.

\begin{definition}
Given $f:\R\p n \to \R$, let $\widetilde\nabla f:\R\p n \rightrightarrows \R\p n$ be defined by
$$\widetilde \nabla f(\overline x) =\left\{ \begin{array}{cl}
    \{\nabla f(\overline x)/|\nabla f(\overline x)|\} & \text{if}~ \overline x \in \Omega ~\text{and}~ \nabla f(\overline x) \neq 0, \\
    \{0\} & \text{if}~ f ~\text{is constant near}~ \overline x, \\
    \emptyset & \text{else},
\end{array}\right.$$
where $\Omega$ are the differentiable points of $f$. The normalized subdifferential is $\widehat\nabla f = \cl \widetilde \nabla f$. 
\end{definition}

We can now define normalized gradient descent.

\begin{definition}
    \label{def:NGD}
    Given $f:\R\p n \to \R$, normalized gradient descent is the Euler discretization of $-\widehat\nabla f$, whose trajectories $\{x_k\}_{k\in \N}\subseteq \R\p n$ obey
$$\forall k \in \N, ~~~ x_{k+1} \in x_k - \alpha_k \widehat{\nabla} f(x_k),$$
for some sequence $\{\alpha_k\}_{k\in \N} \subseteq (0,\infty)$.
\end{definition}

At noncritical points of locally Lipschitz functions, the definition agrees with common intuition.

\begin{lemma}
    \label{lemma:normal_Bouligand}
    Let $f:\R\p n \to \R$ be Lipschitz continuous near $\overline{x}\in\R\p n$. If $0 \notin \overline{\nabla} f(\overline{x})$, then
    $\widehat\nabla f(\overline{x}) = \{ v/|v| : v\in \overline{\nabla} f(\overline{x})\}$.
\end{lemma}
\begin{proof}
    Let $u \in \widehat\nabla f(\overline{x})$. If $u = 0$, then there exists $x_k \to \overline{x}$ such that $\widetilde \nabla f(x_k) = \{0\}$, i.e., $f$ is constant near $x_k$, and so $\nabla f(x_k) = 0$. But then $0 \in \overline{\nabla}f(\overline x)$, a contradiction. Thus there exists $\Omega \ni x_k\to \overline x$ such that such that $\nabla f(x_k) \neq 0$ and $\nabla f(x_k)/|\nabla f(x_k)|\to u$. Since $f$ is Lipschitz continuous near $\overline{x}$, $\nabla f(x_k)$ is bounded. By taking a subsequence if necessary, $\nabla f(x_k) \to v\in \overline{\nabla} f(\overline{x})$ with $v\neq 0$. Thus $ u = v/|v|$. Conversely, let $v \in \overline{\partial}f(\overline{x})$. There exists $\Omega \ni x_k \to \overline{x}$ such that $\nabla f(x_k)\to v\neq 0$. Thus $\nabla f(x_k)\neq 0$ eventually and $\widetilde\nabla f(x_k) =\{ \nabla f(x_k)/|\nabla f(x_k)|\}$. Hence $v/|v| \in \widehat\nabla f(\overline{x})$.
\end{proof}

\cref{lemma:normal_Bouligand} implies that
$$ \widehat{\nabla} f(\overline{x}) = \left\{ v \in \mathbb{R}^n : \exists x_k \xrightarrow[\Omega]{} \overline{x}: \nabla f(x_k) \neq 0 ~\land~ \nabla f(x_k)/|\nabla f(x_k)| \rightarrow v\right\}$$
when $f$ is locally Lipschitz and $0\notin\overline{\nabla}f(\overline{x})$. 
The reason why we do not define $\widehat\nabla f$ in this way is because this expression can be empty when $0\in\overline{\nabla}f(\overline{x})$. It also not clear what value to give it instead. We proceed with some basic properties of the normalized subdifferential.

\begin{lemma}
\label{lemma:constant}
    Let $f:\R\p n \to \R$ be Lipschitz near $\overline{x}\in \R\p n$ and $\nabla f(x)=0$ for all $x\in \Omega$ near $\overline{x}$. Then $f$ is constant near $\overline{x}$.
\end{lemma}
\begin{proof}
    We have $\overline{\partial} f(x)=\co \overline{\nabla}f(x)=\{0\}$ for all $x\in \R\p n$ near $\overline{x}$. Fix $x\in \R\p n\setminus\{\overline{x}\}$ near $\overline{x}$. By the Lebourg mean value theorem \cite[Theorem 2.3.7]{clarke1990}, there exists $u\in(x,\overline{x})$ such that $f(x)-f(\overline{x}) = \langle v,x-\overline{x}\rangle$ for all $v \in \overline{\partial}f(u)=\{0\}$. Hence $f(x) = f(\overline{x})$.
\end{proof}

\begin{proposition}
\label{prop:normal_nonempty}
    If $f:\R\p n \to \R$ is locally Lipschitz, then $\widehat\nabla f$ is locally bounded with nonempty values and has a closed graph.
\end{proposition}
\begin{proof}
    The only thing we need to prove is that the values are nonempty. Suppose $f$ is not constant near $\overline{x}\in \R\p n$. Since $f$ is locally Lipschitz, by Rademacher's theorem $\Omega$ is dense in $\R\p n$. By \cref{lemma:constant}, there exists $\Omega \ni x_k\to \overline{x}$ such that $\nabla f(x_k)\neq 0$. Since the unit sphere is compact, $\widetilde\nabla f(x_k) = \nabla f(x_k)/|\nabla f(x_k)|$ has a limit point $v \in (\cl\widetilde\nabla f)(\overline{x}) = \widehat{\nabla} f(\overline{x})\neq \emptyset$. 
\end{proof}

With this definition in place, we now check a key assumption for the theory of d-stability.

\begin{proposition}
    \label{prop:normal_assumption}
    If $f:\R\p n \to \R$ is locally Lipschitz semi-algebraic, then \cref{assume:Ff} holds with $F = -\widehat \nabla f$ and $D=\overline{\partial} f$.
\end{proposition}
\begin{proof}
    \cref{assume:Ff} \ref{item:F} is due to \cref{prop:normal_nonempty} and \cite[Fact 2.4]{josz2025lyapunov}. As for \cref{assume:Ff} \ref{item:f}, the fact that $f$ is a potential and the existence of an isolated critical value are due to the chain rule \cite[Corollary 5.4]{drusvyatskiy2015curves} and the semi-algebraic Morse-Sard theorem \cite[Corollary 9]{bolte2007clarke}. Let $L>0$ denote a Lipschitz constant of $f$ near $\widehat x$. Fix $x \in \R\p n$ and $v = P_{\overline{\partial}f(x)}(0)$. If $0 \in \overline{\partial} f(x)$, then $v=0$ and $\langle u,v\rangle = 0 = -|v|\p 2/L$ for all $u\in -\co\widehat\nabla f(x)$. Assume $0 \notin \overline{\partial} f(x)$. By characterization of the projection, $\langle w-v,0-v\rangle \leq 0$ for all $w \in \overline{\partial} f(x)$. In other words, $\langle -w,v\rangle \leq -|v|\p 2$ and $\langle -w/|w|,v\rangle \leq -|v|\p 2/L$ for all $w \in \overline{\partial} f(\overline x)\subseteq \overline B_L(0)\setminus \{0\}$. 
    
    Let $u\in -\co\widehat\nabla f(x)$. There exist $\lambda_1,\hdots,\lambda_p\geq 0$ and $u_1,\hdots,u_p\in -\widehat\nabla f(x)$ such that $\lambda_1+\cdots+\lambda_p=1$ and $u = \sum_i \lambda_i u_i$. Since $\overline{\nabla}f(x)\subseteq \overline{\partial} f(x) \subseteq \R\p n \setminus \{0\}$, by \cref{lemma:normal_Bouligand} there exists $w_i \in \overline{\nabla}f(x)$ such that $u_i = -w_i/|w_i|$ and $|w_i|\leq L$. Hence
    \begin{equation*}
        \langle u,v\rangle = \left\langle \sum_i \lambda_iu_i,v\right\rangle = \sum_i \lambda_i\langle -w_i/|w_i|,v\rangle \leq \sum_i \lambda_i|v|\p 2/L = -|v|\p 2/L.
    \end{equation*}
    Therefore
    $$\max_{u\in F(x)}\min_{v\in D(x)} \langle u , v \rangle \leq -\inf_{w\in D(x)} \kappa(w)$$
    where $\kappa(w) = |w|\p 2/L$ is continuous and positive definite.
\end{proof}

The following corollary is immediate.

\begin{corollary}
    \label{cor:descent}
    Let $f:\R\p n\to \R$ be locally Lipschitz semi-algebraic and $x:[a,b]\to \R\p n$ be a trajectory of $-\co\widehat\nabla f$. Then  
    $$\forae t\in [a,b], ~~~ (f\circ x)'(t) \leq - d(0,\overline{\partial} f(x(t)))\p 2/L$$ 
    where $L$ is a Lipschitz constant of $f$ on a neighborhood of $\co( x([a,b]))$.
\end{corollary}
\begin{proof}
    This is a consequence of \cref{prop:normal_assumption}.
\end{proof}

The normalized subdifferential, contrary to the Bouligand subdifferential, is in general not a conservative field for $f$ since otherwise $\overline{\partial} f(x) \subseteq \co \widehat\nabla f(x) \subseteq S\p{n-1}$ by \cite[Corollary 1]{bolte2020conservative}. Hence the importance of considering general set-valued dynamics $F$ in \cref{subsec:Lyapunov_stability}. One should refrain from convexifying $\widehat\nabla f$, otherwise it will generally not be possible to find $p$-d-Lyapunov functions. 

We end this section with some facts pertaining to critical points of the normalized subdifferential. For any function and any point, $0\in \widehat\nabla f(\overline x) \implies 0\in \overline{\nabla}f(\overline x)$ but the converse is false (for e.g., $f(x) = x\p 2$). One also has the following.
\begin{proposition}
    \label{prop:critical}
    Suppose $f:\R\p n\to \R$ is Lipschitz continuous near $\overline{x}\in \R\p n$. Then $0\in \co\widehat\nabla f(\overline x) \implies 0\in \overline{\partial} f(\overline x)$.
\end{proposition}
\begin{proof}
    We argue by contradiction and suppose that $v = P_{\overline{\partial} f(\overline x)}(0)\neq 0$. Since $0 \notin \overline{\nabla}f(x)$, by \cref{lemma:normal_Bouligand}, there exist $\lambda \in \Delta\p{n-1}$ and $v_i \in \overline{\nabla} f(\overline x)$ such that $0 = \sum_i \lambda_i v_i/|v_i|$. By characterization of the projection, $\langle 0-v,w-v\rangle \leq 0$ for all $w \in \overline{\partial} f(\overline x)$, i.e., $\langle v, w \rangle \geq |v|\p 2$. This yields the contradiction $0=\langle v, \sum_i \lambda_i v_i/|v_i| \rangle = \sum_i \lambda_i \langle v,  v_i \rangle /|v_i| \geq \sum_i \lambda_i |v|\p 2 /|v_i| = |v|\p 2\sum_i \lambda_i/|v_i|>0$.
\end{proof}

Again, the converse is false in \cref{prop:critical} (for e.g., $f(x) = x\p 3$). Also, $0\in \co\widehat\nabla f(\overline x) \nRightarrow 0\in \partial f(\overline x)$ (for e.g., $f(x)=-|x|$).

\begin{proposition}[Fermat's rule]
    \label{prop:Fermat}
    Let $f:\R\p n \to \R$ be locally Lipschitz. If $\overline{x}\in \R\p n$ is a local extremum of $f$, then $0 \in \co \widehat\nabla f(\overline{x})$.
\end{proposition}
\begin{proof}
    Suppose $0 \notin \co \widehat\nabla f(\overline{x})$. By \cref{prop:normal_nonempty} and Carathéodory's theorem, $\co \widehat\nabla f(\overline{x})$ is a nonempty closed convex set. Let $v= P_{\co \widehat\nabla f(\overline{x})}(0) \neq 0$. The characterization of the projection yields $\langle0-v,w-v\rangle \leq 0$ for all $w \in \co \widehat\nabla f(\overline{x})$. In particular, $\langle w, v/|v| \rangle \geq |v|>0$ for all $w \in \widehat\nabla f(\overline{x})$. Since $\gph \widehat\nabla f$ is closed by \cref{prop:normal_nonempty}, there exists $\epsilon>0$ such that $\langle w,u\rangle\geq |v|/2$ for all $x\in B_\epsilon(\overline{x})$, $w\in \widehat\nabla f(x)$, and $u\in B_\epsilon(v/|v|)$. After possibly reducing $\epsilon$, $f(x)\geq f(\overline{x})$ for all $x\in B_\epsilon(\overline{x})$, assuming $\overline{x}$ is a local minimum of $f$ (the proof is similar if it is a local maximum).
    
    Let $u \in B_\epsilon(v/|v|)\cap S\p {n-1}$ and $t \in(0,\epsilon]$.
    By the Lebourg mean value theorem \cite[Theorem 2.3.7]{clarke1990}, there exist $a\in (\overline{x},\overline{x}-tu)$ and $b\in \overline\partial f(a)$ such that $0 \leq f(\overline x-tu)-f(\overline x) = \langle b , -tu\rangle = -t\langle b , u \rangle \leq 0$. Indeed, if $b\neq 0$, then there exist $\lambda_1,\hdots,\lambda_p\geq 0$ and $b_1,\hdots,b_p\in \overline\nabla f(a)\setminus\{0\}$ such that $\lambda_1+\cdots+\lambda_p=1$ and $b = \sum_i \lambda_i b_i$. It is easy to see that $b_i/|b_i|\in\widehat\nabla f(a)$. Thus $\langle b_i/|b_i|,u\rangle\geq |v|/2$ and $\langle b , u\rangle = \sum_i \lambda_i \langle b_i , u\rangle\geq 0$. As a result, $f$ is constant on a full dimensional cone near $\overline{x}$, on which $\widetilde\nabla f(x) = \{0\}$. Thus $0\in (\cl \widetilde\nabla f)(\overline x)=\widehat\nabla f(\overline{x})$, a contradiction.
\end{proof}

\begin{lemma}
    \label{lemma:nonzero}
    Let $f:\R\p n\to \R$ be locally Lipschitz semi-algebraic and $\overline{x}\in\R\p n$ be such that $\inte [f=f(\overline{x})]=\emptyset$. Then $0 \notin \widehat\nabla f(x)$ for all $x\in \R\p n$ near $\overline{x}$.
\end{lemma}
\begin{proof}
    Suppose there is $x_k \to \overline{x}$ such that $0 \in \widehat\nabla f(x_k)$. Then there exists $y_k\p i \to x_k$ such that $\widetilde\nabla f(y_k\p i) = \{0\}$, and in particular, $\overline{\partial} f(y_k\p i)=\{0\}$. Since $f$ is locally Lipschitz semi-algebraic, there are no Clarke critical values in a vicinity of $f(\overline{x})$ by \cite[Corollary 9]{bolte2007clarke}. Since $y_k\p k\to \overline{x}$ and thus $f(y_k\p k)\to f(\overline{x})$, one has $f(y_k\p k) = f(\overline{x})$ eventually. But then $y_k\p k \in \inte[f=f(\overline{x})]=\emptyset$, a contradiction.
\end{proof}


\subsection{Effect of discretization}
\label{subsec:Effect of discretization}


In order to understand the effect of discretizing normalized subdifferential dynamics on stability, we will rely on stratification theory \cite{mather1970notes,trotman2020stratification}. 
A $C^k$ stratification of a subset $S$ of $\R\p n$ with $k\in\N\p*$ is a finite partition $\mathcal{X}$ of $S$ into $C^k$ embedded submanifolds of $\R\p n$ such that for all $X,Y \in \mathcal{X}$, if $X \cap \overline{Y} \neq \emptyset$, then $X \subseteq \overline{Y}\setminus Y$. A stratification is semi-algebraic if each element of $\mathcal{X}$, called stratum, is semi-algebraic. It is compatible with a family of subsets of $\R\p n$ if each subset is a union of strata.

A pair of submanifolds $(X,Y)$ of $\mathbb{R}^n$ fulfills the Whitney-(a) condition at $\overline{x} \in X$ if 
$$d(T_xX,T_yY) = o(1)$$ 
for $x\in X$ and $y\in Y$ near $\overline{x}$.
It satisfies the Verdier condition at $\overline{x} \in X$ if 
$$d(T_xX,T_yY) = O(|x-y|)$$ 
for $x\in X$ and $y\in Y$ near $\overline{x}$. Accordingly, a Whitney-(a) (respectively Verdier) stratification of a subset of $\R\p n$ is one in which every pair of strata $(X,Y)$ such that $X \subseteq \overline{Y}\setminus Y$ satisfies the Whitney-(a) (respectively Verdier) condition. They yield variational stratifications.

Let $f:\mathbb{R}^n\rightarrow \overline{\mathbb{R}}$ and $X$ be a submanifold of $\R\p n$ such that $f|_X$ is $C\p 1$. The covariant gradient of $f$ at $\overline{x}\in X$ is defined by $\nabla_X f(\overline{x})=P_{T_{\overline{x}}X}(\nabla \bar{f}(\overline{x}))$ where $\bar{f}$ is any $C^1$ smooth function defined on a neighborhood $U$ of $\overline{x}$ in $\mathbb{R}^n$ and that agrees with $f$ on $U \cap X$. 

\begin{definition}
    \label{def:projection_formulae}
    Let $f:\mathbb{R}^n\rightarrow \eR$ and $X$ be a submanifold of $\R\p n$ such that $f|_X$ is $C\p 1$. We say that $f$ satisfies the projection formula at $\overline{x}\in X$ along $X$ if
\begin{equation*}
    P_{T_{\overline{x}}X} \partial f(\overline{x}) \subseteq \{\nabla_X f(\overline{x})\} ~~~ \text{and}~~~ P_{T_{\overline{x}}X} \partial f^\infty(\overline{x})= \{0\}.
\end{equation*}
We say that $f$ satisfies the perturbed projection formula at $\overline{x}$ along $X$ if
\begin{gather*}
  \forall v\in\partial f(y),\quad  |P_{T_xX}(v)-\nabla_X f(x)| = O(|x-y|\sqrt{1+|v|^2}) \\ \text{and} \\  \forall w\in\partial^\infty f(y),\quad |P_{T_xX}(w)| = O(|x-y||w|)
\end{gather*}
for $x \in X$ and $y \in \R\p n$ near $\overline{x}$.
\end{definition}

A $C^k$ variational Whitney (resp. Verdier) stratification of a function $f:\R^n\to \eR$ is a  $C^k$ Whitney (resp. Verdier) stratification of $\dom  f$ such that $f$ is $C^k$ on each stratum and the projection formula (resp. perturbed projection formula) holds at all $\overline{x}\in \dom  f$. 

If a function $f:\R\p n\to \eR$ is lower semicontinuous (resp. locally Lipschitz continuous on its domain) and semi-algebraic, then it admits a semi-algebraic variational Whitney (resp. Verdier) stratification compatible with any finite family of semi-algebraic subsets of $\R\p n$ \cite[Lemma 8]{bolte2007clarke} (resp. \cite[Theorems 3.6, 3.30]{davis2025active} and \cite[Theorem 1]{bianchi2023stochastic}). For our purposes, we adapt the projection formulae in \cref{def:projection_formulae} to the normalized subdifferential.

\begin{definition}
    \label{def:normalized_projection_formulae}
    Let $f:\mathbb{R}^n\rightarrow \eR$ and $X$ be a submanifold of $\R\p n$ such that $f|_X$ is $C\p 1$. Let $\overline{x}\in X$ and assume $X\subseteq [f=f(\overline{x})]$. We say that $f$ satisfies the normalized projection formula at $\overline{x}$ along $X$ if
\begin{equation*}
    P_{T_{\overline{x}}X} \widehat\nabla f(\overline{x}) \subseteq \{0\}.
\end{equation*}
We say that $f$ satisfies the normalized perturbed projection formula at $\overline{x}$ along $X$ if
$$\forall u\in\widehat\nabla f(y),\quad  |P_{T_xX}(u)| = O(|x-y|)$$
for $x \in X$ and $y \in \R\p n$ near $\overline{x}$.
\end{definition}

Understanding discretization requires us to analyze the continuous-time dynamics, for which we need an extra definition.

\begin{definition}
    \label{def:desingularizer}
    Given $f:\R\p n\to\R$ and $X\subseteq\R\p n$, $\psi:\R_+\to\R_+$ is a desingularizer of $f$ on $X$ if $\psi$ is a concave diffeomorphism such that
$$
\forall x \in X\setminus (\overline{\partial}\widetilde f)\p{-1}(0), ~~~ d(0,\overline{\partial} (\psi \circ \widetilde f)(x)) \geqslant 1.
$$ 
where $\widetilde f(x) = d(f(x),V)$ and $V$ is the set of Clarke critical values of $f$ in $\overline{X}$ if it is nonempty, otherwise $V=\{0\}$.
\end{definition}
If $f$ is locally Lipschitz semi-algebraic and $X$ is bounded, then $f$ admits a desingularizer on $X$ by the Kurdyka-\L{}ojasiewicz inequality \cite{kurdyka1998gradients,bolte2007clarke}. Our analysis begins by noticing that the normalized subdifferential is invariant under desingularization.

\begin{lemma}
    \label{lemma:desingularizer}
    If $f:\mathbb{R}^n \rightarrow \mathbb{R}$ is nonnegative and $\psi:\R_+\to\R_+$ is an increasing diffeomorphism, then $\widetilde\nabla f = \widetilde\nabla(\psi \circ f)$.
\end{lemma}
\begin{proof}
    $f$ is constant near $x$ iff $\psi \circ f$ is constant near $x$, in which case $\widetilde\nabla f(x) = \widetilde\nabla (\psi\circ f)(x) = \{0\}$. $f$ is differentiable and $\nabla f(x)\neq 0$ iff $\psi \circ f$ is differentiable and $\nabla (\psi \circ f) \neq 0$, in which case $\widetilde\nabla f(x) = \nabla f(x)/|\nabla f(x)| = \psi'(f(x))\nabla f(x)/|\psi'(f(x))\nabla f(x)| = \nabla(\psi\circ f)(x)/|\nabla(\psi\circ f)(x)| = \widetilde\nabla (\psi\circ f)(x)$. Indeed, if $f$ is differentiable and $\nabla f(x)\neq 0$, then $f(x)>0$, otherwise $x$ is a local minimum of $f$. Thus $\psi \circ f$ is differentiable and $\nabla (\psi \circ f)(x) = \psi'(f(x)) \nabla f(x) \neq 0$ since $\psi'(f(x))>0$. Conversely, if $\psi \circ f$ is differentiable and $\nabla (\psi \circ f) \neq 0$, then $f(x)>0$, otherwise $x$ is a local minimum of $f$ and thus of $\psi \circ f$. Hence $f = \psi\p{-1} \circ (\psi \circ f)$ is differential at $x$ and $0 \neq \nabla (\psi \circ f)(x) = \psi'(f(x)) \nabla f(x)$, so that $\nabla f(x)\neq 0$.
\end{proof}

The invariance enables one to deduce a sufficient condition for the normalized projection formula to hold.

\begin{lemma}
\label{lemma:normalized_projection_desingularizer}
    Let $f:\mathbb{R}^n \rightarrow \mathbb{R}$ be nonnegative and $\psi:\R_+\to\R_+$ be a desingularizer of $f$ near $\overline{x}$. 
    If $\psi\circ f$ satisfies the projection formula at $\overline{x}$ along $X \subseteq [f=0]$, then $f$ satisfies the normalized projection formula at $\overline{x}$ along $X$. 
\end{lemma}
\begin{proof}
By \cref{lemma:desingularizer}, $\widehat\nabla f = \widehat\nabla g$ where $g=\psi \circ f$. Let $u\in \widehat\nabla f(\overline{x}) = \widehat\nabla g(\overline{x}) = (\cl \widetilde\nabla g)(\overline{x})$. There exists $\gph \widetilde\nabla g \ni (x_k,u_k)\to(\overline{x},u)$. First, if $g$ is eventually constant near $x_k$, then $u_k = 0$ and $u=0$. Second, suppose that eventually $g$ is differentiable at $x_k$ and $\nabla g(x_k)\neq 0$.
Then $\nabla g(x_k) \in \partial g(x_k)$ and $u_k = \nabla g(x_k)/|\nabla g(x_k)|$. If $\nabla g(x_k)$ is bounded, then $\nabla g(x_k)\to v \in \partial g(\overline{x})$ up to a subsequence, by outer semi-continuity of $\partial g$. Since $|\nabla g(x_k)|\geq 1$, $|v|\geq 1$. Hence $u = v/|v|$ and $P_{T_{\overline{x}}X}(u) = P_{T_{\overline{x}}X}(v/|v|) = P_{T_{\overline{x}}X}(v)/|v| = \nabla_X f(\overline{x})/|v| = 0$. Otherwise, up to a subsequence, $\nabla g(x_k)\to \infty$. Then $u\in \partial\p \infty g(\overline{x})$ and so $P_{T_{\overline{x}} X}(u)=0$. Third, if neither of the first or second cases holds, then $u_k\in \widetilde\nabla g(x_k) = \emptyset$ eventually holds, which is impossible.
\end{proof}

The following result allows us to make use of the Kurdyka-\L{}ojasiewicz inequality.

\begin{lemma}
\label{lemma:absolute}
    A function $f:\R\p n\to\R$ satisfies the normalized projection formula at $\overline{x}$ along $X$ iff $|f-f(\overline{x})|$ satisfies the normalized projection formula at $\overline{x}$ along $X$.
\end{lemma}
\begin{proof}
Without loss of generality, assume that $f(\overline{x}) = 0$.
    
($\Longrightarrow$) This is due to the fact that for all $x\in\R\p n$,
    $$\widetilde \nabla |f|(x) =\left\{ \begin{array}{cl}
    \sign(f(x))\widetilde\nabla f(x) & \text{if}~ |f|~\text{is differentiable at}~ x ~\text{and}~ \nabla |f|(x) \neq 0, \\
    \{0\} & \text{if}~ f ~\text{is constant near}~ x, \\
    \emptyset & \text{else}.
\end{array}\right.$$
Indeed, suppose $|f|$ is differentiable at $x$ and $\nabla |f|(x)\neq 0$. If $f(x) \neq 0$, then $f$ is differentiable at $x$ and $\nabla f(x) \neq 0$, and $\nabla|f|(x) = \sign(f(x))\nabla f(x)$. If $f(x)=0$, then $x$ is a local minimum of $|f|$ and so $\nabla|f|(x)=0$, a contradiction.

($\Longleftarrow$) This is due to the fact that for all $x\in\R\p n$,
$$\widetilde \nabla f(x) =\left\{ \begin{array}{cl}
    \sign(f(x))\widetilde\nabla |f|(x) & \text{if}~f~\text{is differentiable at}~ x,~ \nabla f(x) \neq 0, ~\text{and}~ f(x)\neq 0, \\
    \{0\} & \text{if}~ f ~\text{is constant near}~ x, \\
    \emptyset & \text{else}.
\end{array}\right.$$
Indeed, suppose $f$ is differentiable at $x$, $\nabla f(x) \neq 0$, and $f(x) = 0$. If $|f|$ is differentiable at $x$, then $\nabla|f|(x) = 0$ since $x$ is a local minimum of $|f|$. It follows that $\nabla f(x) = 0$, a contradiction.
\end{proof}

The next result is very similar to \cite[Lemma 5.2]{drusvyatskiy2015curves}. We state it and prove it for the reader's convenience.

\begin{lemma}
    \label{lemma:tangent}
    Let $x:[a,b]\to \R\p n$ be differentiable almost everywhere and $\cX$ be a countable family of embedded submanifolds of $\R\p n$. Then $$\forae t \in [a,b],~\forall X\in \cX,~~~ x(t) \in X \implies x'(t)\in T_{x(t)}X.$$
\end{lemma}
\begin{proof}
    Let $\Phi:\cX\rightrightarrows \R$ be defined by $\Phi(X) = \{ t \in [a,b] : x(t) \in X\}$. Recall that a discrete subset of a seperable metric space is countable \cite[Section 30, Example 2]{munkres2000topology}. Thus the set of isolated points of $\Phi(X)$ is countable for each $X \in \cX$. Since $\cX$ is countable, the set of all such isolated points is countable. Hence for almost every $t\in[a,b]$, $x(\cdot)$ is differentiable at $t$ and, for all $X\in \cX$, if $x(t)\in X$, then $t$ is accumulation point of $\Phi(X)$. In that case, there is $t_k\in \Phi(X)\setminus \{t\}$ such that $t_k\to t$, so $(x(t_k)-x(t))/(t_k-t)\to x'(t)\in T_{x(t)}X$ using \cite[Example 6.8]{rockafellar2009variational}.
\end{proof}

When dealing with continuous-time dynamics, we must take the convex hull of the normalized subdifferential in order to guarantee the existence of trajectories. It thus behooves us to find a normalized projection formula for the convexified mapping.

\begin{lemma}
\label{lemma:projection_normal}
    Let $f:\R\p n\to \R$ be a locally Lipschitz function such that the projection formula holds at $\overline{x}$ along $X$. If $0\notin\overline{\nabla} f(\overline{x})$, then $P_{T_{\overline{x}}X}(\co\widehat\nabla f(\overline{x})) \subseteq [|\overline{\nabla} f(\overline{x})|_+\p{-1},|\overline{\nabla} f(\overline{x})|_-\p{-1}]\nabla_X f(\overline{x})$.    
\end{lemma}
\begin{proof}
    Let $u \in \co\widehat\nabla f(\overline{x})$. By \cref{lemma:normal_Bouligand}, there exist $\lambda_1,\hdots,\lambda_p\geq 0$ and $v_1,\hdots,v_p\in \overline{\nabla}f(\overline{x})$ such that $u=\sum_i \lambda_i v_i/|v_i|$ and $\lambda_1+\cdots+\lambda_p = 1$. Since $P_{T_{\overline{x}}X}(v_i) = \{\nabla_X f(\overline{x})\}$, we have $P_{T_{\overline{x}}X}(u) = \sum_i \lambda_i P_{T_{\overline{x}}X}(v_i)/|v_i| = \{\gamma\nabla_X f(\overline{x})\}$ where $\gamma = \sum_i \lambda_i/|v_i| \leq \sum_i \lambda_i/|\overline{\nabla} f(\overline{x})|_- = 1/|\overline{\nabla} f(\overline{x})|_-$ and likewise for $|\overline{\nabla} f(\overline{x})|_+$.
\end{proof}

We now ready to prove the key lemma of this section, which will actually also be useful in a later section when we consider conserved quantities.

\begin{lemma}
    \label{lemma:reparametrization}
    Let $f:\mathbb{R}^n \rightarrow \mathbb{R}$ be locally Lipschitz semi-algebraic and $x:[a,b]\rightarrow \R\p n$ be a curve such that $0 \notin \overline{\partial} f(x(t))$ for all $x\in [a,b]$. Then $x(\cdot)$ is a trajectory of $-\co\widehat\nabla f$ iff it is a trajectory of $-\overline{\partial} f$ up to reparametrization.
\end{lemma}
\begin{proof}
    Since $f$ is locally Lipschitz semi-algebraic, by \cite[Theorems 3.6, 3.30]{davis2025active} it admits a variational Verdier stratification $\cX$. Let $\Phi:\R\p n \rightrightarrows \R\p n$ be defined by $\Phi(x) = T_xX$ where $X$ is the stratum $X\in \cX$ containing $x$. Also, let $\xi:\R\p n\to\R\p n$ be defined by $\xi(x) =\nabla_X f(x)$ where $X$ is the stratum containing $x$. It is measurable since it is Borel measurable. Indeed, $\gph \xi$ is Borel measurable as a countable union of smooth submanifolds of $\R\p {2n}$. 
    
    ($\Longrightarrow$) Suppose $x(\cdot)$ is a trajectory of $-\co\widehat\nabla f$. For almost every $t\in[a,b]$, $x'(t)\in -\co \widehat{\nabla} f(x(t))$ and $x'(t) \in \Phi(x(t))$ by \cref{lemma:tangent}. The assumption on $x(\cdot)$ implies that $0 \notin \overline\nabla f(x(t))$ for all $x\in [a,b]$. By \cref{lemma:projection_normal}, there exists $\gamma:[a,b]\to \R\p n$ such that $|\overline{\nabla}f(x(t))|_+\p{-1}\leq \gamma(t)\leq |\overline{\nabla}f(x(t))|_-\p{-1}$ and a null set $N\subseteq \R$ such that
    $$\forall t\in [a,b]\setminus N, ~~~ x'(t) = P_{\Phi(x(t))}(x'(t)) = -\gamma(t)\xi(x(t)).$$
    Let $c = \inf \{|\xi(x(t))|:t\in[a,b]\setminus N\}$ and suppose $c=0$. Then there exists $t_k\in[a,b]\setminus N$ such that $|x'(t_k)|\leq \gamma(t_k)|\xi(x(t_k))|\to 0$. Up to a subsequence, $t_k\to\overline t\in[a,b]$ and so $0 \in \co\widehat\nabla f(x(\overline{t}))$ by \cref{prop:normal_nonempty}. Thus $0 \in \overline{\partial} f(x(\overline{t}))$ by \cref{prop:critical}, a contradiction.

    Since $\gamma(t) = |x'(t)|/|\xi(x(t))|$ for all $t\in[a,b]\setminus N$, $\gamma$ is measurable. It is also integrable as $0\leq \gamma(t)\leq \max\{ |\overline{\nabla} f(x(s))|_-\p{-1}:s\in[a,b]\}$ for all $t\in[a,b]$. Consider the change of variables $\varphi:[a,b]\to [0,T]$ defined by $\varphi(t) = \int_a\p t \gamma(s)ds$ where $T = \int_a\p b \gamma(s)ds\in(0,\infty)$. Since it is an increasing homeomorphism, we can define $\overline{x} = x\circ\varphi\p{-1}$. For all $t\in[a,b]\setminus N$, with $\overline{t} = \varphi(t)$ we have $\overline{x}'(\overline{t}) = (\varphi\p{-1})'(\overline{t})x'(\varphi\p{-1}(\overline{t})) = \varphi'(t)\p{-1}x'(t) = -\gamma(t)\p{-1}x'(t) = - \xi(x(t)) = -\xi(\overline{x}(\overline{t}))\in - \overline{\partial} f(\overline{x}(\overline{t}))$.

    ($\Longleftarrow$) Suppose $x(\cdot)$ is a trajectory of $-\overline\partial f$. As $\widehat\nabla f$ has a closed graph, it is outer semicontinuous \cite[Theorem 5.7]{rockafellar2009variational}. Together with the fact that it is has closed values, it must be measurable \cite[Exercise 14.9]{rockafellar2009variational}. Measurability is preserved by taking the convex hull \cite[Exercise 14.12]{rockafellar2009variational}, so $\co\widehat\nabla f$ is measurable. Consequently, $-\co \widehat{\nabla} f \circ x$ is measurable. Since it also has closed values by the Carathéodory theorem, it admits a measurable selection $y:[a,b]\to\R\p n$ by \cite[Corollary 14.6]{rockafellar2009variational}. On the other hand, for almost every $t\in[a,b]$, $y(t)\in \Phi(x(t))$ by \cref{lemma:tangent}. By \cref{lemma:projection_normal}, there exists $\gamma:[a,b]\to \R\p n$ such that $|\overline{\nabla}f(x(t))|_+\p{-1}\leq \gamma(t)\leq |\overline{\nabla}f(x(t))|_-\p{-1}$ and
    $$\forae t\in[a,b], ~~~ y(t) = P_{\Phi(x(t))}(y(t)) = -\gamma(t)\xi(x(t)).$$  
    
    Since $1/\gamma(t) = |\xi(x(t))|/|y(t)|$, $1/\gamma$ is measurable. It is also integrable as $0\leq 1/\gamma(t)\leq \max\{ |\overline{\nabla} f(x(s))|_+:s\in[a,b]\}$ for all $t\in[a,b]$. Consider the change of variables $\psi:[a,b]\to [0,T]$ defined by $\psi(t) = \int_a\p t ds/\gamma(s)$ where $T = \int_a\p b ds/\gamma(s)\in(0,\infty)$. It yields a new curve $z = x\circ \psi\p{-1}$. For almost every $t\in[a,b]$, with $s = \psi(t)$ we have $z'(s) = (\psi\p{-1})'(s)x'(\psi\p{-1}(s)) = -\psi'(t)\p{-1}x'(t) = -\gamma(t)\xi(x(t)) = y(t) \in -\co\widehat\nabla f(x(t)) = - \co\widehat\nabla f(z(s))$.
\end{proof}

The above lemma is wrong if Clarke critical points are allowed, as shown by $f(x) = x\p 3$. We hence consider this case below (if $(f\circ x)(t)$ is constant, then $x(t)$ is Clarke critical by \cref{cor:descent}).

\begin{lemma}
    \label{lemma:constant_trajectory}
    Let $f:\R\p n\to\R$ be locally Lipschitz semi-algebraic and $x:[a,b]\to\R\p n$ be a trajectory of $-\co\widehat\nabla f$. If $f\circ x$ is constant, then so is $x$.
\end{lemma}
\begin{proof}
    Without loss of generality, $f(x(a))=0$. Since $f$ is locally Lipschitz semi-algebraic, $|f|$ admits a desingularizer $\psi:\R_+\to\R_+$ near $x([a,b])$. Since $\psi\circ |f|$ is lower semicontinuous and semi-algebraic, by \cite[Lemma 8]{bolte2007clarke} it admits a variational Whitney stratification $\cX$ compatible with $[f=0]$, i.e., such that $[f=0]$ is a union of strata. By \cref{lemma:tangent}, for almost every $t \in[a,b]$, there exists $X \in \cX\cap[f=0]$ such that $x'(t)\in T_{x(t)}X$ and thus $$x'(t) = P_{T_{x(t)}X}(x'(t)) \in -P_{T_{x(t)}X}(\co\widehat\nabla f(x(t))) = -P_{T_{x(t)}X}(\co\widehat\nabla |f|(x(t))) = \{0\}$$ by \cref{lemma:normalized_projection_desingularizer} and \cref{lemma:absolute}. Since $x(\cdot)$ is absolutely continuous, $x(t)-x(a) = \int_a\p t x'(s)ds = 0$ for all $t\in[a,b]$.
\end{proof}

Normalized subdifferential dynamics satisfy the same length formula as Clarke subdifferential dynamics \cite[Proposition 7]{josz2023global}, even though they don't generate the same trajectories. 

\begin{proposition}
    \label{prop:length}
    Let $f:\mathbb{R}^n \rightarrow \mathbb{R}$ be locally Lipschitz semi-algebraic and $X$ be a bounded subset of $\mathbb{R}^n$. Assume that $f$ has at most $m\in \mathbb{N}^*$ critical values in $\overline{X}$. Let $\psi$ be a desingularizer of $f$ on $X$.
    If $x:[a,b]\rightarrow X$ is a trajectory of $-\co\widehat\nabla f$, then 
$$\frac{1}{2m}\int_a^b |x'(t)|dt \leq \psi\left(\frac{f(x(a))-f(x(b))}{2m}\right).$$
\end{proposition}
\begin{proof}
    This is due to \cref{lemma:reparametrization}, \cref{lemma:constant_trajectory}, and \cite[Proposition 7]{josz2023global}.
\end{proof}

We can now make the connection between stability (in continuous time) and d-stability for normalized subdifferential dynamics.

\begin{theorem}
    Let $f:\R\p n \to \R$ be locally Lipschitz semi-algebraic and $\overline{x}\in \R\p n$. The following hold:
    \begin{enumerate}[label=\rm{(\roman{*})}]
        \item $\overline{x}$ is stable for $-\co\widehat \nabla f$ iff it is asymptotically stable for $-\co\widehat \nabla f$ iff it is a local minimum of $f$. \label{item:stable}
        \item If $\overline{x}$ is d-stable for $-\widehat \nabla f$, then $\overline{x}$ is stable for $-\co\widehat \nabla f$. \label{item:d-stable}
    \end{enumerate}
\end{theorem}
\begin{proof}
    \ref{item:stable} This can be proved in a similar flavor to \cite[Theorem 3]{absil2006stable} using \cref{cor:descent} and \cref{prop:length}. As for \ref{item:d-stable}, if $\overline{x}$ is d-stable for $-\widehat \nabla f$, then it is a local minimum of $f$. This can be proved using the same arguments as in \cite[Theorem 1]{josz2023lyapunov} modulo \cite[Lemma 3.1]{josz2025lyapunov}. One then concludes using \ref{item:stable}.
\end{proof}

\subsection{Implicit regularizers}
\label{subsec:Implicit regularizers}

We begin with a necessary condition for being a d-Lyapunov function with normalized gradient descent. The rest of this section is devoted to sufficient conditions.

\begin{proposition}
\label{prop:dL_necessary}
    Let $f:\R\p n\to\R$ be $C\p 1$ semi-algebraic.
    If $g:\R^n\to\eR$ is d-Lyapunov for $-\widehat\nabla f$ and Lipschitz continuous near $\overline{x}$, then $g$ is Lyapunov for $-\nabla f$ near $\overline{x}$. 
\end{proposition}
\begin{proof}
    By \cite[Proposition 3.1]{josz2025lyapunov}, $g$ is a Lyapunov function for $-\widehat\nabla f$ on a neighborhood $U$ of $\overline{x}$. Without loss of generality, $f(\overline{x})=0$. By the semi-algebraic Morse-Sard theorem \cite[Corollary 9]{bolte2007clarke}, there are no critical points in $U \setminus [f=0]$ after possible shrinking $U$. Since $f$ is $C\p 1$, $\widehat\nabla f(x) =\{\nabla f(x)/|\nabla f(x)|\}$ for all $x\in U \setminus [f=f(\overline{x})]$. Let $x:I\to U$ be a trajectory of $-\nabla f$. By the chain rule, $(f\circ x)'(t) = \langle \nabla f(x(t)),x'(t)\rangle = -|\nabla f(x(t))|\p 2$ for all $t\in I$. Consider the intervals $I\p + = \{t\in \ I:f(x(t))> 0\}$ and $I\p - = \{t\in \ I:f(x(t))< 0\}$, as well as the corresponding restrictions $x\p +,x\p -$. Similar to \cref{lemma:constant_trajectory}, $\{t\in I: f(x(t))=0\}$ is at most a singleton since $f$ is $C\p 1$ semi-algebraic. By \cref{lemma:reparametrization} $x\p +,x\p -$ are trajectories of $-\co\widehat\nabla f = -\widehat\nabla f$ up to reparametrization, and so $g\circ x\p +,g\circ x\p -$ are decreasing. By continuity of $g$, $g\circ x$ is decreasing. 
\end{proof}

The general function class in \cref{prop:normal_assumption} shows that analyzing d-stability is purely reduced to the search of a d-Lyapunov function. This is in contrast to an earlier instability result that we proposed \cite[Theorem 6.3]{josz2025subdifferentiation}. In theory, one could prove implicit regularization in a ReLU neural network (which fails to be regular, has level sets with nonempty interior, etc.) as long as an appropriate quantity is found. While no geometric assumptions on the objective function are made in this work (aside from semi-algebraic), they can sometimes help to verify that a function is d-Lyapunov. 

It is helpful to recall a simple fact. In order to obtain a concrete statement regarding the projection on an embedded submanifold, one needs to jump through all sorts of hoops: either by combining \cite[Theorem 3.2, 2.6, Theorem 3.8, Theorem 4.1]{dudek1994nonlinear} or by appealing to
strong amenability \cite[10.23 Definition (b)]{rockafellar2009variational}, prox-regularity \cite[13.31 Exercise, 13.32 Proposition]{rockafellar2009variational}, local closedness \cite[p. 28]{rockafellar2009variational}, and a truncated normal cone \cite[Theorem 1.3]{poliquin2000local}. We detailed the second route in the proof of \cite[Theorem 2.9]{josz2024sufficient}. Here, we've decided to write a separate statement with a full proof in the Appendix.

\begin{fact}
\label{fact:projection}
    Let $M$ be a $C\p 2$ embedded submanifold of $\R\p n$ and $\overline{x}\in M$. Consider the map $E:NM\to \R\p n$ defined by $E(p,v)=p+v$. For all $\rho>0$ sufficiently small, there exists a neighborhood $U$ of $(\overline{x},0)$ in $NM$ such that $E(U) = B_\rho(\overline{x})$, $E|_U:U\to E(U)$ is a $C\p 1$ diffeomorphism, and $$\forall (p,v) \in U, ~~ (P_M \circ E)(p,v) = \{p\} \subseteq B_{2\rho}(\overline{x}).$$
\end{fact}

The normalized perturbed projection formula prevents premature convergence to the solution set.

\begin{proposition}
\label{prop:dist}
    If $f:\R\p n\to \R$ satisfies the normalized perturbed projection formula at $\overline{x}\in \R\p n$ along $M=[f=f(\overline{x})]$, then 
$$\exists \overline\alpha,\rho>0:~ \forall \alpha \in (0,\overline\alpha],~ \forall x \in B_\rho(\overline x),~ \forall u \in -\widehat\nabla f(x)\setminus \{0\}, ~~~ d(x+\alpha u,M)+d(x,M) \geq \alpha/2.$$
\end{proposition}
\begin{proof}
There are $c,r>0$ such that $|P_{T_y M}(u)| \leq c|x-y|$ for all $x\in B_r(\overline{x})$, $y\in M \cap B_r(\overline{x})$, and $u \in \widehat\nabla f(x)$. 
Since $M$ is a $C\p 2$ embedded submanifold, by \cref{fact:projection} the projection is singled-valued and Lipschitz continuous with constant $L>0$ on $B_{r'}(\overline{x})$ for some $r'\in (0,r/2)$ with $P_M(B_{r'}(\overline{x})) \subseteq B_{2r'}(\overline{x})$.
Set $\rho = \min\{r',1/(Lc)\}/2$ and $\overline{\alpha} = \rho$. For all $x\in B_{\rho}(\overline{x})$, $u \in -\widehat\nabla f(x)\setminus \{0\}$, and $y = P_M(x)$, we have
$$d(x+\alpha u,M) = |x+\alpha u - P_M(x+\alpha u)| \geq \alpha - |x - P_M(x+\alpha u)|$$
and
\begin{align*}
    |x - P_M(x+\alpha u)| & \leq |x -P_M(x)| + |P_M(x) - P_M(x+\alpha u)| \\
    & = d(x,M) + |P_M(x+\alpha P_{N_y M}(u)) - P_M(x+\alpha u)|\\
    & \leq d(x,M) + L\alpha|P_{N_y M}(u)-u| \\
    & = d(x,M) + L\alpha|P_{T_y M}(u)| \\
    & \leq d(x,M) + L\alpha c |x-P_M(x)| \\
    & = (1+Lc\alpha)d(x,M)
\end{align*}
where $P_M(x)=P_M(x+\alpha P_{N_y M}(u))$ holds by \cite[Theorem 2.9, Step 4]{josz2024sufficient}.
Thus $d(x+\alpha u,M) \geq \alpha - (1+Lc\alpha)d(x,M)$ and 
$d(x+\alpha u,M)+ d(x,M) \geq \alpha [1 - Lc\hspace{.3mm} d(x,M)] \geq \alpha/2$ since $d(x,M) \leq 1/(2Lc)$.
\end{proof}

Another useful assumption is due to Mordukovitch and Ouyang \cite[Definition 3.1 (i)]{mordukhovich2015higher}.

\begin{definition}
\label{def:submetric}
A mapping $F:\mathbb{R}^n\rightrightarrows\mathbb{R}^m$ is metrically $\tau$-subregular at $(\overline{x},\overline{y}) \in \gph F$ with $\tau>0$ if 
$d(x,F^{-1}(\overline{y})) = O( d(\overline{y},F(x))^\tau)$
for $x\in \R\p n$ near $\overline{x}$.
\end{definition}

Together with the normalized perturbed projection formula, metric subregularity allows one to convert a weak decrease condition into a $(p,2)$-d-Lyapunov function. 

\begin{proposition}
    \label{prop:p2dL}
    Suppose a locally Lipschitz semi-algebraic function $f:\R\p n\to \R$ satisfies the normalized perturbed projection formula at a local minimum $\overline{x}\in \R\p n$ of $f$ along $M=[f=f(\overline{x})]$ with $\inte M=\emptyset$ and $\partial f$ is metrically $\tau$-subregular at $(\overline{x},0)$. Let $g:\R^n\to\eR$, $a,b\in[1,\infty)$, $\overline{\alpha},\rho,\omega>0$ be such that $g$ is continuous on $B_\rho(\overline{x})$ and
$$\forall \alpha\in (0,\overline{\alpha}], ~ \forall x \in B_\rho(\overline{x})\setminus M,~ \forall u\in \widehat\nabla f(x), ~ \exists v\in \partial f(x): ~~ g(x-\alpha u) - g(x) \leq - \omega\alpha\p a |v|\p b.$$
Then $g$ is $(a+b/\tau,2)$-d-Lyapunov near $\overline{x}$.
\end{proposition}
\begin{proof}
We actually have
$$\forall \alpha\in (0,\overline{\alpha}], ~ \forall x \in B_\rho(\overline{x}),~ \forall u\in \widehat\nabla f(x), ~ \exists v\in \partial f(x): ~~ g(x-\alpha u) - g(x) \leq - \omega\alpha\p a |v|\p b.$$
Indeed, let $\alpha \in(0,\overline{\alpha}]$, $x \in B_\rho(\overline{x})$, and $u\in \widehat{\nabla}f(x)$. If $u = 0$, then $0 \in \overline{\nabla} f(x)\subseteq \partial f(x)$, in which case the property holds with $v = 0$. Otherwise, there exists $\gph \widetilde\nabla f \ni(x_k,u_k)\to (x,u)$ where $u_k$ is eventually nonzero. But then $f$ is differentiable at $x_k$ with $\nabla f(x_k)\neq 0$, so in particular $x_k\notin M$.
By assumption, there exists $v_k \in \partial f(x_k)$, such that $g(x_k-\alpha u_k) - g(x_k) \leq - \omega\alpha\p a |v_k|\p b$. Since $\partial f$ is locally bounded and has a closed graph, $v_k \to v$ for some $v \in \partial f(x)$. By continuity of $g$, passing to the limit yields $g(x-\alpha u) - g(x) \leq - \omega\alpha\p a |v|\p b$.


By semi-algebraicity and metric $\tau$-subregularity, there exist $c,r>0$ such that, for all $x\in B_r(\overline{x})$, $d(x,M) = d(x,(\partial f)\p{-1}(0)) \leq c\hspace{.3mm}d(0,\partial f(x))\p \tau \leq c |v|\p \tau$ for all $v\in \partial f(x)$. After possibly reducing $\rho$ and $\overline{\alpha}$, we may assume that $\rho \leq r$, $\overline{\alpha}\leq \rho/2$, and $0\notin\widehat\nabla f(x)$ for all $x\in B_\rho(\overline{x})$ by \cref{lemma:nonzero}. Let $\{\alpha_k\}_{k\in\N}\subseteq (0,\overline{\alpha}]$ and $\{x_k\}_{k\in\N}\subseteq \R\p n$ be such that $x_{k+1}\in x_k - \alpha_k \widehat\nabla f(x_k)$ for all $k\in\N$. Assume $x_k \in B_{\rho/2}(\overline{x})$ for some $k\in \N$, so that $x_{k+1} \in B_\rho(\overline{x})$. By the previous paragraph, there exist $v_k\in \partial f(x_k)$ and $v_{k+1}\in \partial f(x_{k+1})$ such that
\begin{align*}
    g(x_{k+2})-g(x_k) & = g(x_{k+2})-g(x_{k+1}) + g(x_{k+1})-g(x_k) \\ & \leq -\omega \alpha_{k+1}\p a |v_{k+1}|\p b -\omega \alpha_k\p a |v_k|\p b \\
    & \leq -\omega c\p {1/\tau} [\alpha_{k+1}\p a d(x_{k+1},M)\p {b/\tau} + \alpha_k\p a d(x_k,M)\p {b/\tau}] \\
    & \leq -\omega c\p {1/\tau} \min\{\alpha_{k+1},\alpha_k\}\p a [d(x_{k+1},M)\p {b/\tau} + d(x_k,M)\p {b/\tau}] \\
    & \leq -\omega c\p {1/\tau} \min\{\alpha_{k+1},\alpha_k\}\p a \min\{1,2\p{1-b/\tau}\} [d(x_{k+1},M) + d(x_k,M)]\p {b/\tau}\\
    & \leq -\omega c\p {1/\tau} \min\{\alpha_{k+1},\alpha_k\}\p a \min\{1,2\p{1-b/\tau}\} (\alpha_k/2)\p {b/\tau} \\
    & \leq -\omega c\p {1/\tau} \min\{2\p{-b/\tau},2\p{1-2b/\tau}\} \min\{\alpha_{k+1},\alpha_k\}\p {a+b/\tau}
\end{align*}
using $x\p p+y\p p\geq \min\{1,2\p {1-p}\}(x+y)\p p$ for any $x,y,p>0$ and \cref{prop:dist}.
\end{proof}

At this point, it is incumbent on us to find practical sufficient conditions for the normalized perturbed projection formula to hold, complementing the more theoretical one in \cref{lemma:normalized_projection_desingularizer}. The first harnesses symmetry, while the second exploits a composite structure. They both use the following simple lemma.

\begin{lemma}
    \label{lemma:normalized_projection}
    Let $f:\mathbb{R}^n\rightarrow \eR$ and $X$ be a submanifold of $\R\p n$ such that $f|_X$ is $C\p 1$. Let $\overline{x}\in X$ and assume $X\subseteq [f=f(\overline{x})]$.
    Consider the following statements:
    \begin{enumerate}[label=\rm{(\roman{*})}]
        \item $\exists c,r>0, ~\forall x \in X\cap B_r(\overline{x}), ~ \forall y \in B_r(\overline{x}),~\forall v\in\partial f(y),~ |P_{T_xX}(v)| \leq c|x-y||v|$. \label{item:|v|}
        \item $\exists c,r>0, ~\forall x \in X\cap B_r(\overline{x}), ~ \forall y \in B_r(\overline{x}),~\forall u\in\widetilde\nabla f(y),~  |P_{T_xX}(u)| \leq c|x-y|$. \label{item:tilde}
        \item $f$ satisfies the normalized perturbed projection formula at $\overline{x}$ along $X$. \label{item:hat}
    \end{enumerate}
    One has
    \ref{item:|v|} $\implies$ \ref{item:tilde} $\Longleftrightarrow$ \ref{item:hat}.
\end{lemma}
\begin{proof}
    \ref{item:|v|} $\implies$ \ref{item:tilde}: If $f$ is differentiable at $y$ and $v = \nabla f(y)\neq 0$, then $\widetilde\nabla f(y) = \{v/|v|\}$ and $|P_{T_xX}(v/|v|)|\leq c|x-y|$. If $f$ is constant near $y$, then $\widetilde\nabla f(y) = \{0\}$ and $|P_{T_xX}(0)| = 0 \leq c|x-y|$. \ref{item:tilde} $\Longleftrightarrow$ \ref{item:hat} is straightforward.
\end{proof}

Below, $G$ denotes a Lie group acting on $\R\p n$.

\begin{proposition}
    \label{prop:projection_symmetry}
    If $f:\R^n\to\R$ is $G$-invariant and Lipschitz continuous near $\overline x\in\R\p n$, then $f$ satisfies the normalized perturbed projection formula at $\overline{x}$ along $G\overline{x}$. 
\end{proposition}
\begin{proof}
    By the orbital perturbed projection formula \cite[Proposition 3.6]{josz2025subdifferentiation}, there are $c,r>0$ such that
    $|P_{T_xGx}(v)|\leq c|x-y||v|$ for all $x\in G\overline{x}\cap B_r(\overline{x})$, $y\in B_r(\overline{x})$, and $v\in\partial f(y)$. One now concludes by \cref{lemma:normalized_projection}.
\end{proof}

One could in fact deduce another sufficient condition without assuming Lipschitz continuity using the orbital perturbed projection formula in \cite[Proposition 3.4]{josz2025subdifferentiation}.

\begin{proposition}
\label{prop:projection_composition}
    Let $f=g\circ F$ where $g:\R\p m \to \R_+$ is lower semicontinuous and positive definite, and $F:\R\p n \to\R\p m$. Suppose $F$ is $C\p 2$ near a point $\overline{x}\in M = [f = 0]$ with image $\overline{y} = F(\overline{x})$. Let $U \subseteq \R\p n$ be an open set containing $\overline{x}$ and $V\subseteq \R\p m$ be a $C\p 2$ embedded submanifold near $F(\overline{x})$ such that $F(U) \subseteq V$. Let $\widetilde F:U\to V$ denote the corresponding restriction of $F$. If $d\widetilde F_{\overline{x}}$ is surjective and $F\p{-1}(\overline{y}) \subseteq U$, then $f$ satisfies the normalized perturbed projection formula at $\overline{x}$ along $M$.
\end{proposition}
\begin{proof}
Let $\psi:V\to \R\p m$ be a chart centered at $F(\overline{x})$. Since $F(x) = (\psi \circ \widetilde F)(x)$ for all $x \in U$, by the chain rule of differential geometry \cite[Proposition 3.6]{lee2012smooth} $dF_x = d\psi_{F(x)} \circ d\widetilde F_x$. As $d\widetilde F_{\overline{x}}$ is surjective, after possibly shrinking $U$, $d\widetilde F_x$ is surjective for all $x \in U$. Thus $M_x = [\widetilde F= \widetilde F(x)]\cap U = [F = F(x)]\cap U$ is an embedded submanifold for all $x\in U$ by 
\cite[Corollary 5.13]{lee2012smooth}
and $T_x M_x = \ker d\widetilde F_x = \ker dF_x$ by \cite[Corollary 5.38]{lee2012smooth}. Since $g$ is positive definite, $M_{\overline{x}} = [f = 0] \cap U = M \cap U$.
A chain rule of variational analysis \cite[Theorem 10.6]{rockafellar2009variational} yields
    $\partial f(x) \subseteq dF_x\p * \partial g(F(x)) \subseteq \Im dF_x\p * = (\ker dF_x)\p \perp = (T_x M_x)\p \perp$ for all $x \in U$.
    Thus $P_{T_xM_x}(v) = 0$ for all $v\in \partial f(x)$ and
\begin{align*}
    |P_{T_xM_x}(v)| & = |P_{T_xM_x}(v)-P_{T_yM_y}(v)| \\
    & \leq |P_{T_xM_x}-P_{T_yM_y}| |v| \\
    & = d(T_x M_x,T_y M_y)|v| \\
    & = d(\ker dF_x,\ker dF_y)|v| \\
    & = d(\im dF_x\p *,\im dF_y\p *)|v| \\
    & \leq C|dF_x\p *-dF_y\p *||v| \\
    & \leq CL|x-y||v|.
\end{align*}
The constant $C>0$ exists because $dF_x\p *$ has constant rank \cite[Section 3]{wedin1972perturbation}. The constant $L>0$ exists because $F$ is $C\p 1$, so one can apply the mean value theorem. One now concludes by \cref{lemma:normalized_projection}.
\end{proof}

\subsection{Conservation, flatness, and symmetry}
\label{subsec:conserved}

In this section, we explore the connection between stability and flatness, then branch out to conservation and symmetry. Let us recall a recently proposed definition of flatness.

\begin{definition}{\normalfont \cite[Definitions 3.1, 3.2]{josz2025flat}}
\label{def:flat}
    Given $f:\R\p n \to \R$, let $\cf:\R\p n \times \R_+ \to \eR$ be defined by
$$\cf(x,r) = \sup_{\overline{B}_r(x)} |f-f(x)|.$$
Consider the preorder on $\R\p n$ defined by
$$x \preceq_f y ~~~ \Longleftrightarrow ~~~  \exists \overline{r}>0: ~ \forall r \in (0,\overline{r}], ~~~ \cf(x,r) \leq \cf(y,r).$$ 
We say $\overline{x}$ is flat if there exists a neighborhood $U$ of $\overline{x}$ in $[f=f(\overline{x})]$ such that $\overline{x}\preceq_f x$ for all $x \in U$. Also, one has the strict preorder $x\prec_f y$ iff $x\preceq_f y$ and $y\npreceq_f x$.
\end{definition}

The following result forges a link between stability and flatness. It is the main result we will be using in the forthcoming numerical experiments. Below, we use regular in the sense of variational analysis \cite[Definition 7.25]{rockafellar2009variational}, which is defined using normal cones and epigraphs, neither of which are needed in this paper so we take a pass (unlike in our previous work \cite[Section 2.4]{josz2025flat}).

\begin{theorem}
\label{thm:implicit_NGD}
Let $f:\R\p n \to \R$ be locally Lipschitz semi-algebraic with $\min f =0$. Let $g:\R\p n \to \eR$ be continuous on its full measure domain with $\min g= 0$ such that $[0<g<\infty]$ is open and $g$ is $C\p {2,2}$ on $[0<g<\infty]$. Suppose 
\begin{enumerate}[label=\rm{(\rm{\roman*})}]
        \item $f+g$ is coercive;
        \item $\forall x\in [0<g<\infty], ~ \forall v \in \overline\nabla f(x), ~ \langle \nabla g(x),v \rangle \geq 0$; \label{item:first_order}
    \item $\forall x\in[f=0]\cap[0<g<\infty],~\forall u\in \widehat\nabla f(x),~\langle \nabla\p 2 g(x) u , u \rangle< 0$.
\end{enumerate}
    Then $\arg\min f+g$ is an almost sure asymptotic 2-attractor for $-\widehat\nabla f$, which is an  asymptotic 2-attractor if $\dom g =\R\p n$. If in addition $g$ is $C\p 4$ on $[0<g<\infty]$, and either
\begin{enumerate}[resume*]
    \item \label{item:reg} $f$ is regular and $\forall x \in [f = 0]\cap [0<g<\infty], ~ \lip f(x) >0$, or
    \item \label{item:D2} $f$ is $D\p 2$ and $\forall x \in [f = 0]\cap [0<g<\infty], ~ \nabla\p 2 f(x) \neq 0$,
\end{enumerate}
then $\arg\min f+g$ contains the flat global minima of $f$.
\end{theorem}
\begin{proof}
Since $\nabla g$ is continuous on the open set $\dom g$, \ref{item:first_order} is equivalent to 
$$\forall x\in \dom g, ~ \forall u \in \widehat\nabla f(x), ~ \langle \nabla g(x),u \rangle \geq 0.$$
This follows from the definitions of $\overline{\nabla}f$ and $\widehat\nabla f$: ($\Longrightarrow$) Let $\gph \widetilde\nabla f\ni(x_k,u_k)\to (x,u)$. If $x_k\in\Omega$ and $v_k =\nabla f(x_k)\neq 0$, then $u_k = v_k/|v_k|$ and $\overline\nabla f(x_k)=\{v_k\}$. Thus $\langle \nabla g(x_k) , u_k \rangle = \langle \nabla g(x_k),v_k/|v_k|\rangle = \langle \nabla g(x_k),v_k\rangle/|v_k|\geq 0$. Of course, $\langle \nabla g(x_k),u_k\rangle \geq 0$ also holds in the other case, i.e., when $u_k = 0$. Passing to the limit yields $\langle \nabla g(x),u\rangle \geq 0$. ($\Longleftarrow$) Let $\Omega \ni x_k\to x$ be such that $\nabla f(x_k)\to v$. If $v_k =\nabla f(x_k) \neq 0$, then $\widehat\nabla  f(x_k) = \{v_k/|v_k|\}$ and so $\langle \nabla g(x_k),v_k/|v_k| \rangle \geq 0$, i.e., $\langle \nabla g(x_k),v_k \rangle \geq 0$. This also holds if $\nabla f(x_k)=0$ of course, and thus $\langle \nabla g(x), v \rangle \geq 0$. 

Attractiveness is due to \cref{prop:normal_assumption}, \Cref{thm:attractor}, and \cref{prop:sufficient_dL2}. Let $x_0 \in [f = 0]\cap [0<g<\infty]$. By assumption, $\langle \nabla\p 2 g(x_0)u,u\rangle < 0$ for all $u\in \widehat\nabla f(x_0)$. Since $\widehat\nabla f(x_0)$ is compact and does not contain $0$, there exists $\omega>0$ such that for all $u\in \widehat\nabla f(x_0)$, $\langle \nabla\p 2 g(x_0)u,u\rangle \leq -2\omega$. Since $\nabla\p 2 g$ is continuous and $\widehat\nabla f$ is locally bounded with a closed graph by \cref{prop:normal_nonempty}, there exists $\rho>0$ such that
$$\forall x \in B_\rho(x_0),~ \forall u\in \widehat\nabla f(x),~~~ \langle \nabla\p 2 g(x)u,u\rangle \leq -\omega.$$ 
After possibly reducing $\rho$, $f$ has no Clarke critical values reached in $B_\rho(x_0)\setminus[f=0]$, so that $0\notin \overline\nabla f(x)$ for all $x\in B_\rho(x_0)\setminus[f=0]$.
By \cref{lemma:normal_Bouligand}, $\widehat\nabla f(x) = \{v/|v|:v\in\overline\nabla f(x)\}$. Thus
$$\forall x \in B_\rho(x_0)\setminus[f=0],~ \forall u\in \overline\nabla f(x),~~~ \langle \nabla\p 2 g(x)u,u\rangle \leq -\omega |u|\p 2.$$ 
Hence the weaker version of \cite[Assumption 4.1]{josz2025flat} described in \cite[Remark 4.7]{josz2025flat} is satisfied.
As $\nabla g$ is continuous on $B_\rho(x_0)$ after possibly reducing $\rho$, by the Cauchy-Peano theorem \cite[Theorem 1.2]{coddington1955theory} the ordinary differential equation 
$$\left\{
\begin{array}{ccc}
     \dot{x} & = & -\nabla g(x),  \\
     x(0) & = & x_0, 
\end{array}
\right.$$
admits a local solution $x:[0,T)\to B_\rho(x_0)$ where $T>0$. By \cite[Corollary 4.6, Remark 4.7]{josz2025flat}, $x(t) \prec_f x(s)$ for all $0 \leq s < t < T$ if either \ref{item:reg} or \ref{item:D2} holds. Thus $x_0$ is not flat.
\end{proof}

We next explore the connection with conservation, which will become especially evident in the numerical examples. Complementing \cref{prop:dL_necessary}, below is a necessary condition for being a d-Lyapunov function under slightly different assumptions.

\begin{proposition}
    \label{prop:dL_necessary_D1}
    Let $f:\R\p n\to\R$ be locally Lipschitz and $g:\R^n\to\eR$ be d-Lyapunov and differentiable on $U\subseteq\R\p n$. Then
    $$\forall x\in U,~\forall v\in \overline{\partial} f(x), ~~~\langle v,\nabla g(x)\rangle\geq 0.$$
\end{proposition}
\begin{proof}
        By \cite[Proposition 3.2]{josz2025lyapunov}, for all $x\in U$ and $u \in -\widehat\nabla f(x)$, $\langle\nabla g(x),u\rangle\leq 0$. If $0\notin\overline{\nabla} f(x)$, then $\widehat\nabla f(x) = \{ v/|v|:v\in \overline{\nabla} f(x)\}$ by \cref{lemma:normal_Bouligand}, and so, for all $v \in \overline\nabla f(x)$, $0\leq \langle\nabla g(x),v/|v|\rangle = \langle\nabla g(x),v\rangle/|v|$, i.e., $\langle \nabla g(x),v\rangle \geq 0$. Passing to the convex hull yields $\langle \nabla g(x),v\rangle \geq 0$ for all $v\in\overline{\partial} f(x)=\co\overline{\nabla}f(x)$. This obviously still holds if $0\in\overline\nabla f(x)$. 
\end{proof}

\cref{prop:dL_necessary_D1} suggests that conserved quantities can be a good choice of d-Lyapunov functions. If one can find sufficiently many of them, then one can hope to check for strict optimality in \cref{thm:stable_point}, as follows. 
\begin{proposition}
\label{prop:strict}
    Suppose $f:\R\p n \to \R$ satisfies the normalized projection formula at $\overline{x}\in \R\p n$ along  $[f= f(\overline{x})]$ and $C:\R^n\to\R^m$ is differentiable at $\overline{x}$. If $\overline{x}$ is a local minimum of $f$ and
    $\R^n= \sp \widehat{\nabla} f(\overline{x})+\im C'(\overline{x})^*$,
    then $\overline{x}$ is a strict local minimum of $f+|C-C(\overline{x})|\p 2/4$.
\end{proposition}
\begin{proof}
Let $M = [f= f(\overline{x})]$.
We argue by contradiction and assume there exists $x_k\to\overline{x}$ such that $x_k\neq \overline{x}$, $f(x_k)=f(\overline{x})$, and $C(x_k)=C(\overline{x})$. By taking a subsequence if necessary, $(x_k-\overline{x})/{|x_k-\overline{x}|}\to w \in T_{\overline{x}}M \subseteq (\sp\widehat\nabla f(\overline{x}))\p \perp$. Indeed, $P_{T_{\overline{x}}M} \widehat\nabla f(\overline{x}) = 0$ and so $\sp\widehat\nabla f(\overline{x})\subseteq (T_{\overline{x}}M)\p \perp$. On the other hand,
 \[   C(x_k)=C(\overline{x})+C'(\overline{x})(x_k-\overline{x})+o(|x_k-\overline{x}|)    \]
 Since $C(x_k)=C(\overline{x})$, dividing both sides by $|x_k-\overline{x}|$ and taking the limit yields $C'(\overline{x})w=0$. Hence $w\in \ker C'(\overline x)=(\im C'(\overline{x})^*)^\perp$. Together, this yields the contradiction $w\in (\sp\widehat\nabla f(\overline{x}))\p \perp\cap(\im C'(\overline{x})^*)^\perp = (\sp\widehat\nabla f(\overline{x})+\im C'(\overline{x})^*)^\perp=(\R^n)^\perp=\{0\}$.
\end{proof}

\cref{prop:strict} suggests that finding a maximal number of independent conserved quantities \cite[Definition 2.18]{marcotte2023abide} can aid in establishing d-stability. If $C$ is $C\p \infty$, independence simply means that $C'(x)\p *$ is injective. The conserved quantities of normalized subdifferential dynamics and Clarke subdifferential dynamics are the same, expect for trajectories passing through saddle points.

\begin{proposition}
    \label{prop:conserved_normalized}
    Let $f:\mathbb{R}^n \rightarrow \mathbb{R}$ be locally Lipschitz semi-algebraic and $c:\R\p n\to \eR$. The following are equivalent:
    \begin{enumerate}[label=\rm{(\roman{*})}]
        \item For any trajectory $x:[0,T)\to\R\p n$, $T\in(0,\infty]$, of $-\co\widehat\nabla f$ devoid of saddle points of $f$, $c\circ x$ is constant.
        \item For any trajectory $x:[0,T)\to\R\p n$, $T\in(0,\infty]$, of $-\overline{\partial} f$ devoid of saddle points of $f$, $c\circ x$ is constant.
    \end{enumerate}
\end{proposition}
\begin{proof}
    This follows from \cref{lemma:reparametrization} and \cref{lemma:constant_trajectory}.
\end{proof}

Conserved quantities can be constructed using linear symmetries \cite[Proposition 5.1]{zhao2023symmetries} \cite[Corollary 5.1]{josz2025subdifferentiation}.
The set of invertible matrices with real coefficients of order $n$, denoted  $\GL(n)$, is a Lie group. The natural action of a Lie subgroup $G$ of $\GL(n)$ on $\R\p n$ is defined by the matrix vector multiplication $G\times \R\p n \ni (g,x)\mapsto gx\in \R\p n$. We let $\fg$ denote the Lie algebra of $G$ and let $\mathrm{s}(\fg)$ be its symmetric elements.

\begin{proposition}
{\normalfont \cite[Proposition 4.9]{josz2025flat}}
\label{prop:conserved}
Let $f:\mathbb{R}^n\rightarrow \mathbb{R}$ be locally Lipschitz and invariant under the natural action of a Lie subgroup $G$ of $\GL(n)$, $C(x) = P_{\mathrm{s}(\mathfrak{g})}(xx^T)$, and $\overline{x} \in \R\p n$.
Let $c(x) = \|C(x)-C(\overline{x})\|_F\p 2/4$. Then 
    $$\forall x \in \R\p n,~ \forall v \in \overline{\partial} f(x), ~~~  \langle \nabla c(x), v \rangle = 0 ~~~\text{and}~~~ \langle \nabla\p 2 c(x) v , v \rangle  = \langle C(x) - C(\overline{x}) , C(v) \rangle_F.$$
\end{proposition}

In \cite[Theorem 6.3]{josz2025subdifferentiation}, we proposed using such a quadratic conserved quantity $C(x)$ to detect instability in the subgradient method with constant step size. We later defined $c(x)$ to analyze flatness \cite{josz2025flat}. On a final note, it is sometimes helpful to do a change of variables when studying stability. Let $\O(n) = \{Q\in\R\p{n\times n}: Q\p T Q = I_n\}$ denote the orthogonal group.

\begin{proposition}
    \label{prop:change}
    Let $U\in \O(n)$ and $f,\widetilde f:\R\p n\to \R$ be such that $\widetilde f(x) = f(Ux)$ for all $x\in\R\p n$. Let $\{x_k\}_{k\in\N}\subseteq\R\p n$ and $\widetilde x_k = U\p T x_k$ for all $k\in\N$. Then
    \begin{enumerate}[label=\rm{(\rm{\roman*})}]
        \item $x$ is a flat minimum of $f$ $~~~\Longleftrightarrow~~~$ $U\p T x$ is a flat minimum of $\widetilde{f}$; \label{item:flat}
        \item $\forall x\in\R\p n,~ \widehat\nabla \widetilde f(x) = U\p T \widehat\nabla f(Ux)$; \label{item:normal_unitary}
        \item $\forall k\in\N,~~~ x_{k+1} \in x_k - \alpha_k\widehat\nabla f(x_k) ~~~\Longleftrightarrow ~~~ \widetilde x_{k+1} \in \widetilde x_k - \alpha_k\widehat\nabla \widetilde f(\widetilde x_k)$. \label{item:update}
    \end{enumerate}
\end{proposition}
\begin{proof}
    \ref{item:flat} For all $x\in\R\p n$ and $r\in \R_+$, 
    $$\accentset{\circ}{\widetilde f}(x,r) = \sup_{y\in B_r(x)} |\widetilde f(y)-\widetilde f(x)| = \sup_{Uy\in B_r(Ux)} |f(Uy)- f(Ux)|= \cf (Ux,r).$$
    Thus the partial order is preserved, i.e., $x\preceq_f y \Longleftrightarrow U\p Tx \preceq_{\widetilde f} U\p T y$. Since $U\p T$ is invertible and thus homeomorphic, the desired result follows. 
    
    \ref{item:normal_unitary} $\widetilde f$ is differentiable at $x$ with $\nabla\widetilde f(x)\neq 0$ iff $f$ is differentiable at $Ux$ with $\nabla f(Ux)\neq 0$, in which case $|\nabla \widetilde f(x)| = |U\p T\nabla f(Ux)|=|\nabla f(Ux)|$ so $\widetilde \nabla \widetilde f(x) = U\p T \widetilde \nabla f(Ux)=\nabla f(Ux)/|\nabla f(Ux)|$. $\widetilde f$ is constant near $x$ iff $f$ is constant near $Ux$, in which case 
    $\widetilde \nabla \widetilde f(x) = U\p T \widetilde \nabla f(Ux)= \{0\}$. Hence $\widetilde\nabla \widetilde f(x) = U\p T \widetilde\nabla f(Ux)$ for all $x\in\R\p n$, and one concludes by taking the closure.
    
    \ref{item:update} We have $x_{k+1} \in x_k - \alpha_k\widehat\nabla f(x_k)$ iff $U\p T x_{k+1} \in U\p T x_k - \alpha_kU\p T \widehat\nabla f(UU\p Tx_k)$ iff $U\p T x_{k+1} \in U\p T x_k - \alpha_k \widehat\nabla\widetilde f(U\p Tx_k)$ iff $\widetilde x_{k+1} \in \widetilde x_k - \alpha_k\widehat\nabla \widetilde f(\widetilde x_k)$.
\end{proof}

\section{Examples}
\label{sec:Examples}


Some examples are in order. We respectively define the sign and Bouligand sign of $x\in\R$ by
\begin{equation*}
    \sign(x)=\left\{\begin{array}{cl}
        x/|x| & \text{if}~ x\neq 0, \\
        \left[-1,1\right] & \text{if}~ x=0,
    \end{array}\right.
    ~~~\text{and}~~~
        \sgn(x)=\left\{\begin{array}{cl}
        x/|x| & \text{if}~ x\neq 0, \\
        \{-1,1\} & \text{if}~ x=0.
    \end{array}\right.
\end{equation*}

\begin{example}
\label{eg:y4}
    The flat minimum $(0,0)$ of $f(x,y) = y\p 2 + x\p 2 y\p 4$ is an asymptotic 4-attractor.
\end{example}
\begin{proof}
Compute $\nabla f(x,y) = 2(xy\p 4,y+2x\p 2 y\p 3)$ and $|\nabla f(x,y)| = 2|y||(xy\p 3,1+2x\p 2 y\p 2)|$. As $d((x,y),[f=0]) = |y| \leq |\nabla f(x,y)|/2$, $\nabla f$ is metrically 1-subregular. Since $$P_{\sp\{(1,0)\}}(\nabla f(x,y)) = xy\p 4 = O(y\p 2) = O(|y||\nabla f(x,y)|) = O(d((x,y),[f=0])|\nabla f(x,y)|),$$ $f$ satisfies the normalized perturbed projection formula at any global minimum along $[f=0]$ by \cref{lemma:normalized_projection}. Moreover, let $g(x,y)=|x|$ and observe that 
    \begin{align*}
        g(x\p +,y\p +) & = |x\p +| \\ 
        & = |x-\alpha xy\p 4/|(xy\p 4, y+2x\p 2 y\p 3)|| \\
        & = |x|(1-\alpha y\p 4/|(xy\p 4, y+2x\p 2 y\p 3)|) \\
        & = g(x,y)(1 - \alpha |y|\p 3/|(xy\p 3, 1+2x\p 2 y\p 2)|) \\
        & \leq g(x,y)(1 - \alpha|\nabla f(x,y)|\p 3/2) \\
        & \leq g(x,y)-g(x,y)|\nabla f(x,y)|\p 3\alpha/2
    \end{align*}
near any point in $[f=0]$. By \cref{prop:p2dL}, $g$ is $(4,2)$-d-Lyapunov near any point in $[f=0]\cap [g>0]$. What's more, $f(x,y)+g(x,y) = y\p 2 + x\p 2 y\p 4 + |x|\geq |x|+|y|$ whenever $|y|\geq 1$, so $f+g$ is coercive. By \cref{thm:stable_set}, $\arg\min f+g = \{(0,0)\}$ is an asymptotic 4-attractor. As for why the origin is the sole flat minimum, see \cite[Example 5.3]{josz2025flat}. 
\end{proof}

In \cref{eg:y4}, $\lambda_1(\nabla\p 2 f(x,0))=0$ for all $x \in \R$, so this quantity is not sufficient to explain implicit regularization.

\begin{figure}[H]
\centering
\begin{subfigure}{.49\textwidth}
  \centering
  \includegraphics[width=1\textwidth]{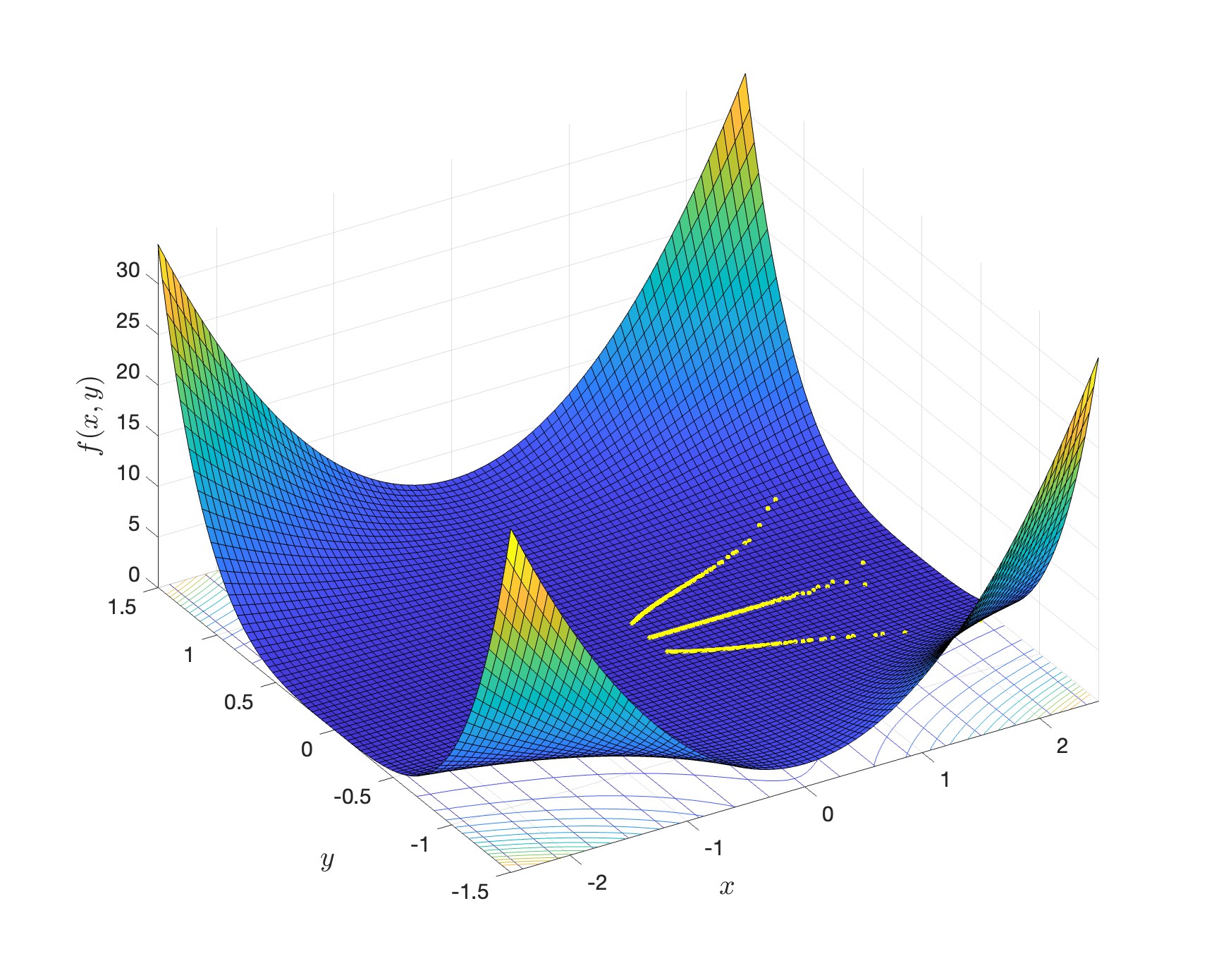}
  \caption{$f(x,y)=y\p 2+x\p 2 y \p 4$.}
  \label{fig:fourth_power}
\end{subfigure}
\begin{subfigure}{.49\textwidth}
\centering
  \includegraphics[width=.8\textwidth]{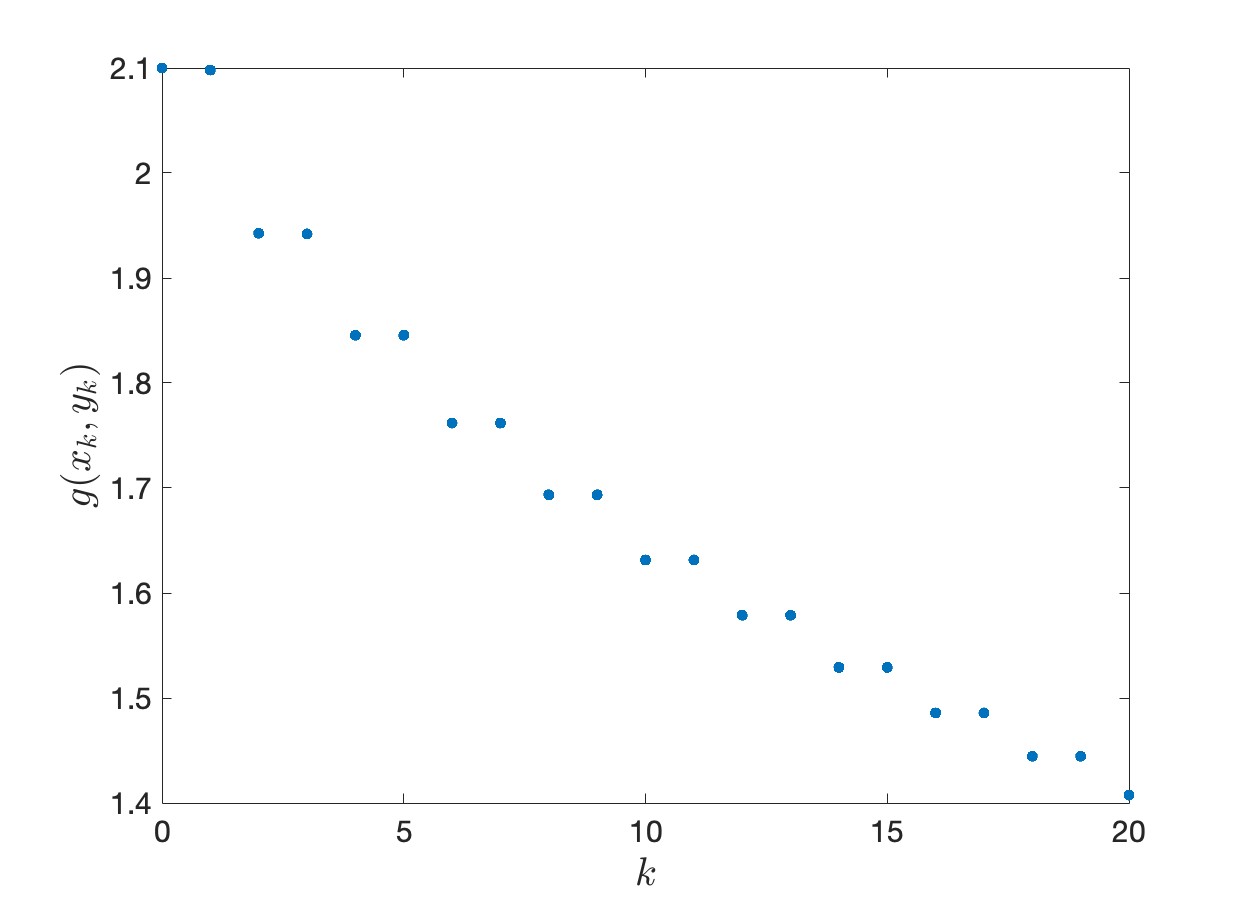}
  \caption{$g(x,y)=|x|$.}
  \label{fig:fourth_power_g}
\end{subfigure}
\caption{Normalized gradient descent with step size $1/(k+1)\p {1/4}$.}
\end{figure}

\begin{example}[Parabola]
    \label{eg:parabola}
    The flat minimum $(0,0)$ of $f(x,y) = (x \p 2 - y)\p 2$ is an asymptotic 2-attractor.
\end{example} 
\begin{proof}
Compute 
$$\nabla f(x,y) = 2(x\p 2 - y) \begin{pmatrix}
    2x \\ - 1
\end{pmatrix} ~~~\text{and}~~~ \widehat \nabla f(x,y) = \frac{\sgn(x\p 2 - y)}{\sqrt{4x\p 2 +1}} \begin{pmatrix}
    2x \\ - 1
\end{pmatrix}.$$
Observe that $f=F^2$ where $F(x,y)=x^2-y$. Thus $\lambda_1(\nabla^2f(x,y))=2|F'(x,y)|^2=8x^2+2$ whenever $F(x,y)=0$ by \cite[Fact 3.26]{josz2025lyapunov}. 
The origin is the unique minimum of $\lambda_1(\nabla^2f)$ on $[f=0]$, which implies that it is the sole flat minimum of $f$ by \cite[Corollary 3.19]{josz2025flat}.
In order to find a conserved quantity, consider the system 
\begin{equation*}
    \left\{
    \begin{array}{rcc}
    \dot{x} & = & 2x, \\
    \dot{y} & = & -1.
    \end{array}
    \right.
\end{equation*}
It admits the unique solution $(x(t),y(t)) = (x(0)e\p{2t},y(0)-t)$, which satisfies $x(t) = x(0)e\p{2y(0) - 2y(t)}$, i.e., $x(t)e\p {2y(t)} = x(0)e\p{2y(0)}$. Thus let $g:\R\p 2\to \R$ be defined by $$g(x,y) = x\p 2 e \p {4y}.$$ 
Since
\begin{equation*}
    \nabla g(x,y) =  2\begin{pmatrix}
        x e\p {4y} \\
        2 x\p 2 e\p {4y}
    \end{pmatrix} ~~~\text{and}~~~
        \nabla\p 2 g(x,y) =  2 \begin{pmatrix}
        e\p {4y} & 4x e\p {4y} \\
        4x e\p {4y} & 8x\p 2 e\p {4y}
    \end{pmatrix},
\end{equation*}
we have $\langle \nabla g(x,y) , \nabla f(x,y) \rangle = 0$ for all $x,y\in \R$ and
\begin{equation*}
\forall (x,y) \in [f=0]\cap[g>0],~\forall u\in \widehat\nabla f(x,y), ~~~ \langle \nabla\p 2 g(x,y) u , u \rangle  = -\frac{8x\p 2e\p {4y}}{4x\p 2 +1}< 0.
\end{equation*}
Moreover, for all $x,y\in \R$,
\[   f(x,y)+g(x,y)=(x^2-y)^2+x^2e^{4y}\geq (x^2-y)^2+x^2(1+4y)=(x^2+y)^2+x^2. \]
Hence $f+g$ is coercive. By \cref{thm:implicit_NGD}, $\arg\min f+g = \{(0,0)\}$ is an asymptotic 2-attractor. 
\end{proof}


\begin{figure}[H]
\centering
\begin{subfigure}{.49\textwidth}
  \centering
  \includegraphics[width=1\textwidth]{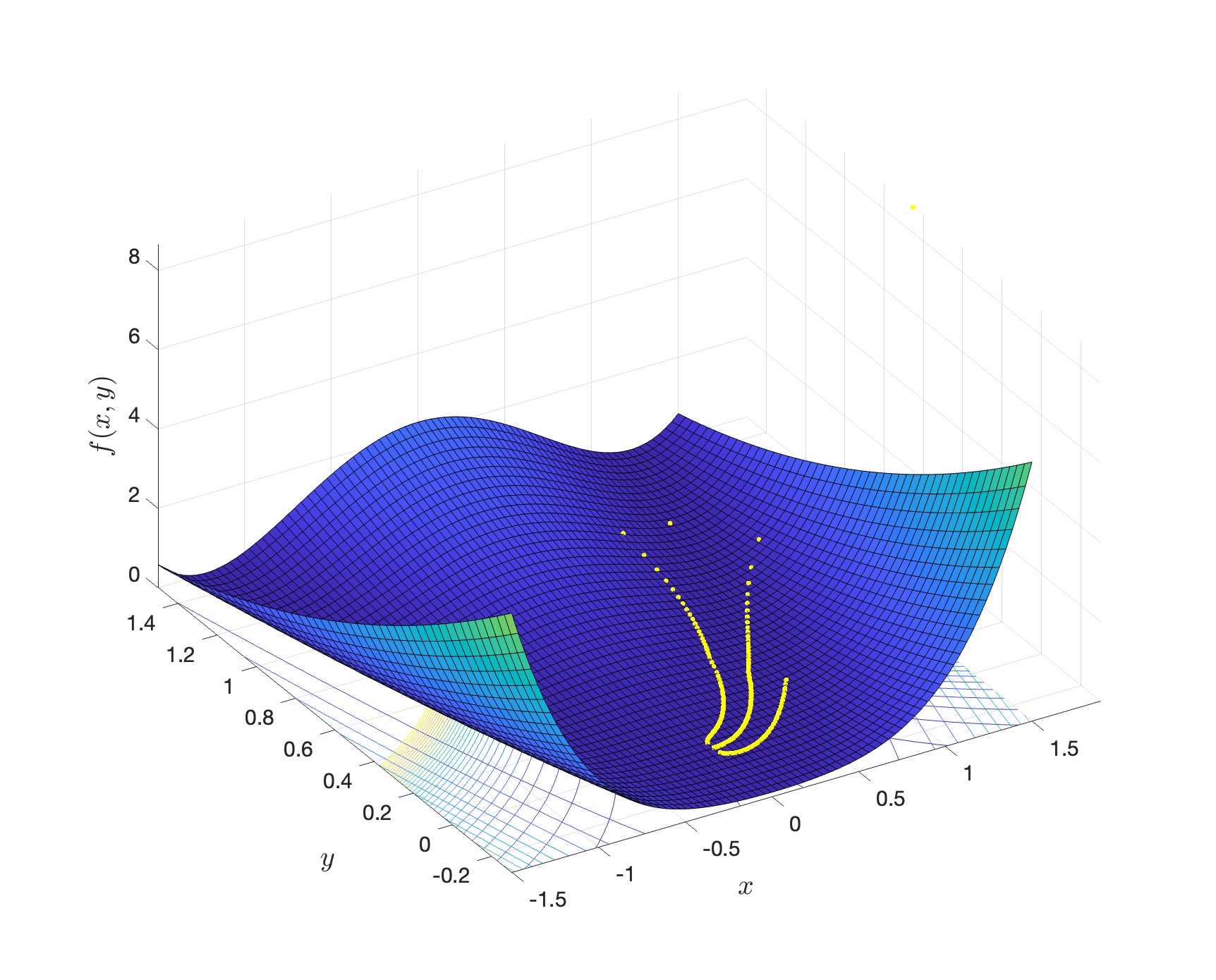}
  \caption{$f(x,y)=(x^2-y)\p 2$.}
  \label{fig:parabola}
\end{subfigure}
\begin{subfigure}{.49\textwidth}
\centering
  \includegraphics[width=.8\textwidth]{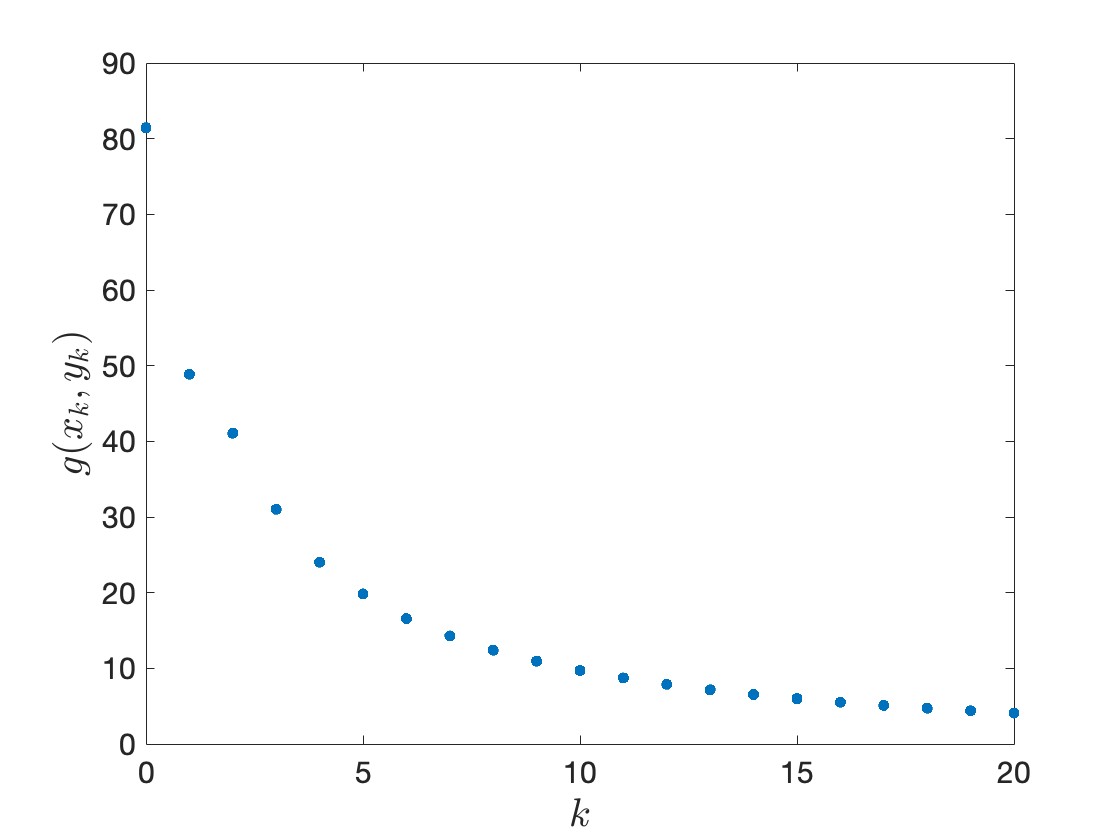}
  \caption{$g(x,y)=x^2e\p{4y}$.}
  \label{fig:parabola_g}
\end{subfigure}
\caption{Normalized gradient descent with step size $1/\sqrt{k+1}$.}
\end{figure}


\begin{example}[Cubic]
\label{eg:cubic}
    The flat minimum $(1/2\p{1/3},-1/2\p{1/3})$ of $f(x,y)=(x\p 3-y\p 3 -1)\p 2$ is an asymptotically 2-stable point contained in the asymptotic 2-attractor $\{(x,(x\p 3-1)\p {1/3}):x\in [0,1]\}$.
\end{example}
\begin{proof}
Compute
    $$\nabla f(x,y) =
    6(x\p 3-y\p 3 -1) 
    \begin{pmatrix}
        x\p 2 \\ - y\p 2    
    \end{pmatrix} ~~~\text{and}~~~ \widehat\nabla f(x,y) =
    \frac{\sgn(x\p 3-y\p 3 -1)}{\sqrt{x\p 4+y\p 4}} 
    \begin{pmatrix}
        x\p 2 \\ - y\p 2    
    \end{pmatrix}.$$
    Since $\nabla f(x,y)=0$ implies $f(x,y)=0$ or $(x,y)=(0,0)$, which a saddle point, every local minimum is a global minimum.
    Observe that $f=F^2$ where $F(x,y)=x^3-y\p 3-1$. Thus $\lambda_1(\nabla^2f(x,y))=2|F'(x,y)|^2=18(x\p 4+y\p 4)$ whenever $F(x,y)=0$ by \cite[Fact 3.26]{josz2025lyapunov}. Consider the Lagrangian $L(x,y,\lambda) = x\p 4+y\p 4 + \lambda(x\p 3 - y\p 3 -1)$. The gradient of the constraint is nonzero for all feasible points, so constraint qualification holds. First-order optimality yields $x=-y=-3\lambda/4$ and so $(x,y)=(1/2\p{1/3},-1/2\p{1/3})$. This point is thus the unique minimum of $\lambda_1(\nabla^2f)$ on $[f=0]$, which implies that it is the sole flat minimum by \cite[Corollary 2]{josz2025flat}.
    
    Consider the system
    \begin{equation*}
    \left\{
    \begin{array}{rcc}
    \dot{x} & = & x\p 2, \\
    \dot{y} & = & -y\p 2.
    \end{array}
    \right.
\end{equation*}
    We have $\dot{x}x\p {-2}+\dot{y}y\p {-2}=0$, hence $-1/x-1/y$ is conserved. Let $g:\R\p 2 \to \R$ be defined by 
    $$ g(x,y) = 
    \left\{\begin{array}{cc}
         \exp\left[-\sgn(x+y)\left(\frac{1}{x}+\frac{1}{y}\right)\right] & \text{if}~ xy\neq 0,\\
         0 & \text{if}~ xy = 0. 
    \end{array}
    \right.
    $$ 
    When $x+y\neq 0$, we have
\begin{equation*}
    \nabla g(x,y) = s g(x,y) \begin{pmatrix}
        x\p {-2} \\
        y\p {-2}
    \end{pmatrix} ~~~\text{and}~~~
        \nabla\p 2 g(x,y) = g(x,y) \begin{pmatrix}
        x\p{-4}-2sx\p{-3} & x\p{-2}y\p{-2} \\
        x\p{-2}y\p{-2} & y\p{-4}-2sy\p{-3}
    \end{pmatrix}
\end{equation*}
where $s=\sgn(x+y)$. Thus
\begin{equation*}
\forall u\in \widehat\nabla f(x,y),~~~ 
    \langle \nabla g(x,y) , u \rangle = 0 ~~~\text{and}~~~
        \langle \nabla\p 2 g(x,y) u , u \rangle  = 
        -\frac{2g(x,y)|x+y|}{x\p 4+y\p 4}<0
\end{equation*}
near any global minimum, aside from $(0,-1)$, $(1/2\p{1/3},-1/2\p{1/3})$, and $(1,0)$. By \cref{prop:sufficient_dL2}, $g$ is $2$-d-Lyapunov near any such point. Observe that $g(x,y)\geq 1$ iff $xy\neq 0 \land \sgn(x+y)\left(1/x+1/y\right)\leq 0$ iff $xy<0$. Hence, if $xy<0$ and $f(x,y)+g(x,y)=1$, then $f(x,y)=0$ and $g(x,y)=1$, i.e., $(x,y) = (1/2\p{1/3},-1/2\p{1/3})$. This point is thus a strict local minimum of $f+g+\delta_{[g\geq 1]}$, and so it is asymptotically 2-stable by \cref{thm:stable_set}. 

Let $h = g \chi_{[g<1]}$. We have $\sup_{[f=0]}h=1$ and $h(x,y)<1$ for all $(x,y)\in[f=0]$. For any $\widetilde h \in (0,1)$, $[f\leq 1]\cap[h\leq\widetilde h]$ is bounded, otherwise there exists $(x_k,y_k)$ with $|(x_k,y_k)|\to\infty$ such that $(x_k\p 3-y_k\p3-1)\p 2\leq 1$ (and hence $|x_k|,|y_k|\to\infty$), $x_ky_k> 0$, and $1/x_k+1/y_k\geq \ln(1/\widetilde h)>0$. This is impossible. One now concludes by \cref{thm:attractor} and \cite[Remarks 3.17 and 3.18]{josz2025lyapunov}.
\end{proof}

\begin{figure}[H]
\centering
\begin{subfigure}{.49\textwidth}
  \centering
  \includegraphics[width=.88\textwidth]{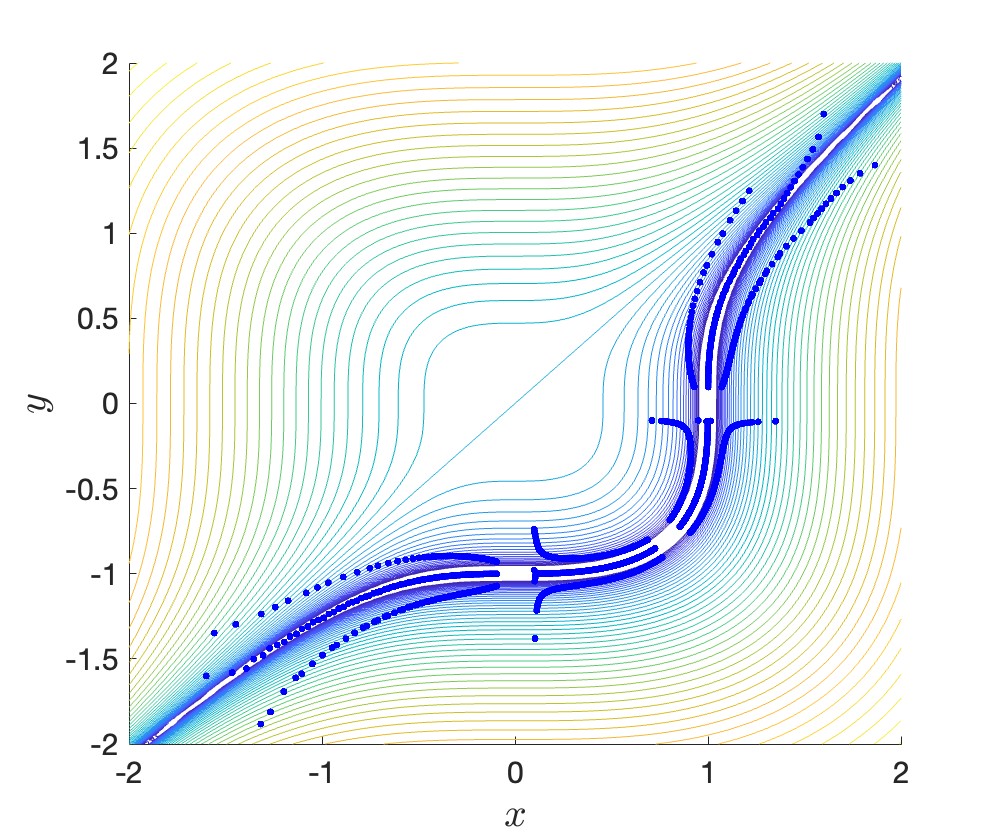}
  \caption{$f(x,y)=(x\p 3-y\p 3 -1)\p 2$.}
  \label{fig:x3y3}
\end{subfigure}
\begin{subfigure}{.49\textwidth}
\centering
  \includegraphics[width=1\textwidth]{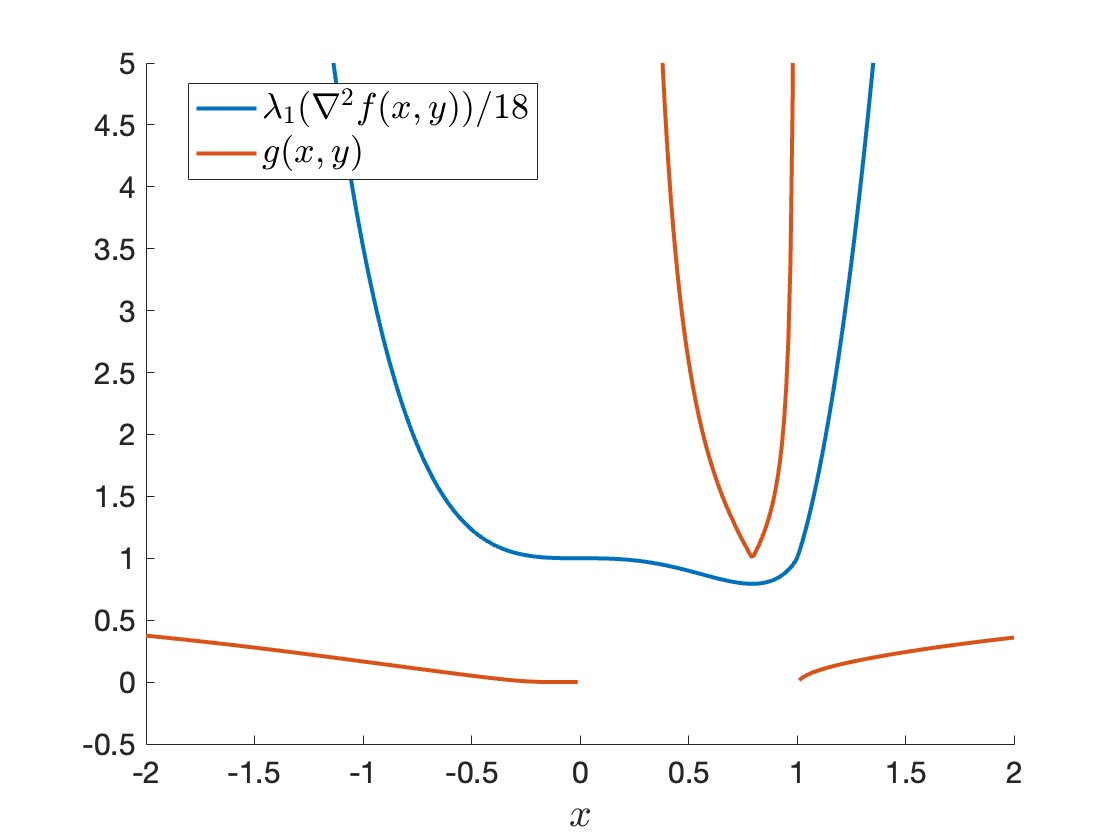}
  \caption{$\lambda_1 (\nabla \p 2 f(x,y)) = 18(x\p 4+|x\p 3-1|\p {4/3})$ and $g(x,y) = \exp[\sgn(1-2x\p 3)(1/x+1/(x\p 3-1)\p{1/3})]$ when $y=(x\p 3-1)\p {1/3}$.}
  \label{fig:x3y3_Hessian}
\end{subfigure}
\caption{Normalized gradient descent with step size $0.4/(k+1)\p {1/4}$ and 4 initial points.}
\end{figure}

\begin{example}[Ellipse, hyperbola]
    \label{eg:ellipse}
    The flat minima of $f(x,y) = (ax^2+by^2-1)\p 2$ are almost sure asymptotic 2-attractors. They are $\pm (0,1/\sqrt{b})$ if $a>b>0$ or $a<0<b$.
\end{example}
\begin{proof}
Compute
$$\nabla f(x,y) = 4(ax\p 2 + by\p 2-1) \begin{pmatrix}
    ax \\ by
\end{pmatrix} ~~~\text{and}~~~ \widehat \nabla f(x,y) = \frac{\sgn(ax\p 2 + by\p 2-1)}{\sqrt{a\p 2x\p 2 +b\p 2 y\p 2}} \begin{pmatrix}
    ax \\ by
\end{pmatrix}.$$
Observe that $f=F^2$ where $F(x,y)=ax^2+by\p 2$. Thus $\lambda_1(\nabla^2f(x,y))=2|F'(x,y)|^2=8(a\p 2x^2+b\p 2 y\p 2)$ whenever $F(x,y)=0$ by \cite[Fact 3.26]{josz2025lyapunov}. Suppose $a>b>0$ or $a<0<b$. Then $\pm (0,1/\sqrt{b})$ are the only minima of $\lambda_1(\nabla^2f)$ on $[f=0]$, which implies that they are the flat global minima of $f$ by \cite[Corollary 3.19]{josz2025flat}. 
Consider the system 
\begin{equation*}
    \left\{
    \begin{array}{rcc}
    \dot{x} & = & ax, \\
    \dot{y} & = & by.
    \end{array}
    \right.
\end{equation*}
It admits the unique solution $(x(t),y(t)) = (x(0) e\p {at},y(0) e\p {bt})$. If $y(0)\neq 0$, then $x(t) \p b / y(t)\p a = x(0) \p b / y(0)\p a$. Thus let $g:\R\p 2\to \eR$ be defined by
$$g(x,y) = \left\{
\begin{array}{cc}
    |x|\p b/|y|\p a & \text{if}~y\neq 0, \\
    \infty & \text{if}~y=0.
\end{array}
\right.
$$ 
Assume $b\geq 3$, so that $g$ is $C\p {2,2}$. Since
\begin{equation*}
    \nabla g(x,y) =  \begin{pmatrix}
        bx|x|\p {b-2}/|y|\p {a} \\
        -ay|x|\p {b} / |y|\p {a+2}
    \end{pmatrix} ~~~\text{and}~~~
        \nabla\p 2 g(x,y) = \begin{pmatrix}
        b(b-1)|x|\p {b-2}/|y|\p {a} & -abxy|x|\p {b-2}/|y|\p {a+2} \\
        -abxy |x|\p {b-2} / |y|\p {a+2} & a(a+1)|x|\p {b} / |y|\p {a+2}
    \end{pmatrix},
\end{equation*}
we have $\langle \nabla g(x,y) , \nabla f(x,y)\rangle = 0$ for all $(x,y)\in\R\p 2$ and
\begin{equation*}
\forall (x,y)\in [f=0]\cap[0<g<\infty],~\forall u\in \widehat\nabla f(x,y),~~~
        \langle \nabla\p 2 g(x,y) u , u \rangle  = -\frac{ab(a-b) g(x,y)}{a\p 2x\p 2+b\p 2y\p 2} < 0.
\end{equation*}
If $a>b>0$, then $f$ is coercive, and $g$ is nonnegative, so $f+g$ is coercive. If $a<0<b$, then $f(x,y)+g(x,y)\geq g(x,y)=|x|^b|y|^{-a}\to\infty$ as $|x|,|y|\to\infty$, and $f(x,y)+g(x,y)\geq f(x,y)\to\infty$ as $|x|\to \infty$ and $y$ is bounded or vice-versa. Thus $f+g$ is coercive. One now concludes by \cref{thm:implicit_NGD}. 

If $b<3$, then one may either use $g\p {3/b}$ as a $C\p {2,2}$ d-Lyapunov function, or consider the function $\breve{f}(x,y) = (9/b\p 2)f(x,y) =[3(a/b)x\p 2 + 3y\p 2-3/b]\p 2$, for which $\widehat\nabla \breve f = \widehat\nabla f$. In any case, it is also possible to proceed without requiring $g$ to be $C\p{2,2}$: for all $(x,y)$ near a point in $\dom g$, for all $s\in \sgn(ax^2+by^2-1)$, and for all sufficiently small $\alpha>0$, we have
\begin{align*}
    g(x\p+,y\p+) &= \frac{|x\p +|\p b}{|y\p +| \p a} = \frac{\left|x-\alpha\frac{sax}{\sqrt{a^2x^2+b^2y^2}}\right|^b}{\left|y-\alpha\frac{sby}{\sqrt{a^2x^2+b^2y^2}}\right|^a}=\frac{|x|^b}{|y|^a}\frac{\left(1-\frac{\alpha sa}{\sqrt{a^2x^2+b^2y^2}}\right)^b}{\left(1-\frac{\alpha sb}{\sqrt{a^2x^2+b^2y^2}}\right)^a} \\
    &=g(x,y)\exp\left[b\ln\left(1-\frac{\alpha sa}{\sqrt{a^2x^2+b^2y^2}}\right)-a\ln\left(1-\frac{\alpha sb}{\sqrt{a^2x^2+b^2y^2}}\right)\right]\\
    &=g(x,y)\exp\left(-\frac{s\p 2(ba^2-ab^2)\alpha^2}{2(a^2x^2+b^2y^2)} +O(\alpha^3)\right) \\
    &\leq g(x,y)\exp\left( -\frac{ab(a-b)}{4(a^2x^2+b^2y^2)}\alpha\p 2\right)             
\end{align*}
since $\ln(1+t) = t-t\p 2/2+O(t\p 3)$.
Thus $g$ is $2$-d-Lyapunov near any point in $[0<g<\infty]$.
\end{proof}


\begin{figure}[H]
\centering
\begin{subfigure}{.49\textwidth}
  \centering
  \includegraphics[width=1\textwidth]{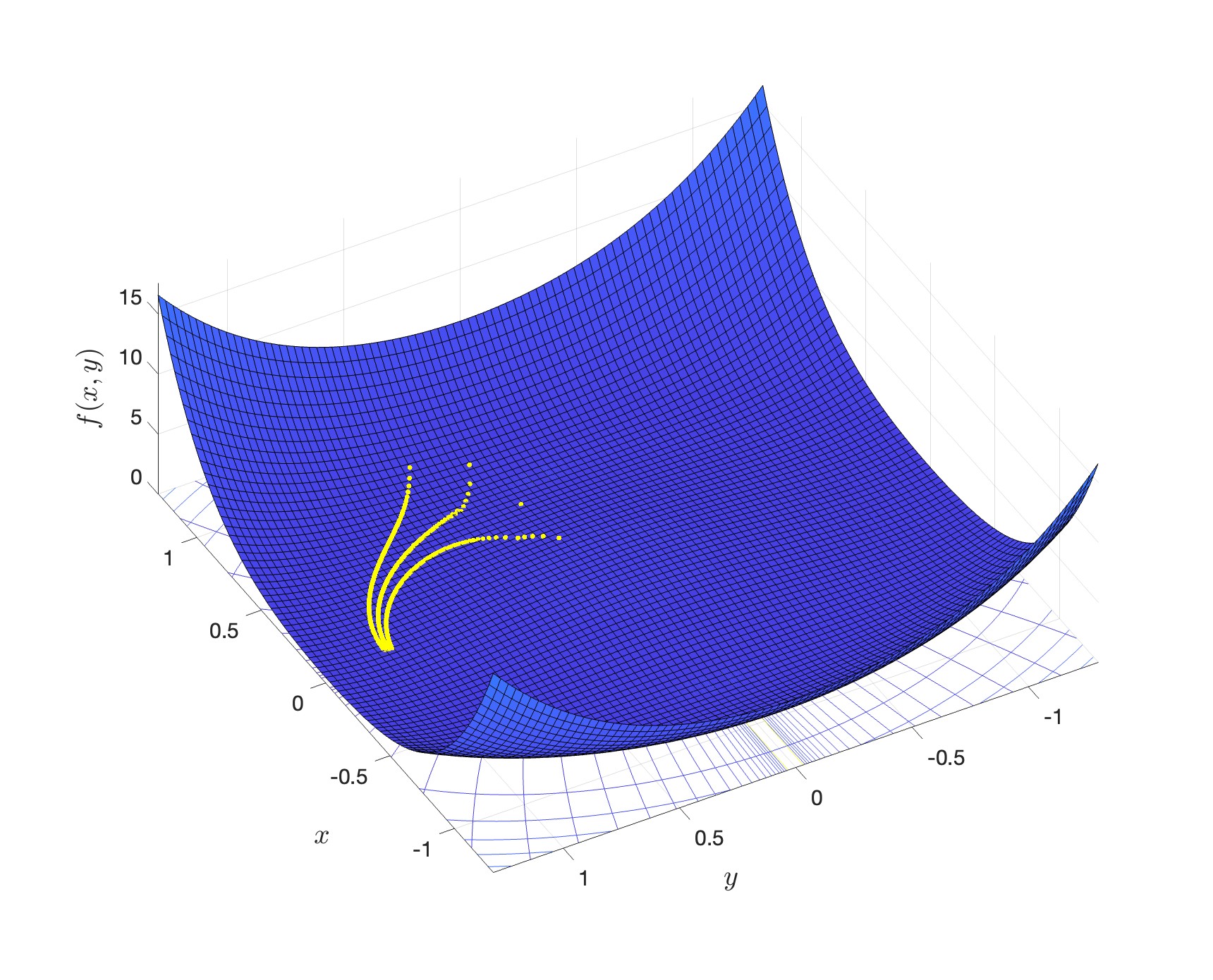}
  \caption{$f(x,y)=(2x^2+y^2-1)\p 2$.}
  \label{fig:ellipse}
\end{subfigure}
\begin{subfigure}{.49\textwidth}
\centering
  \includegraphics[width=1\textwidth]{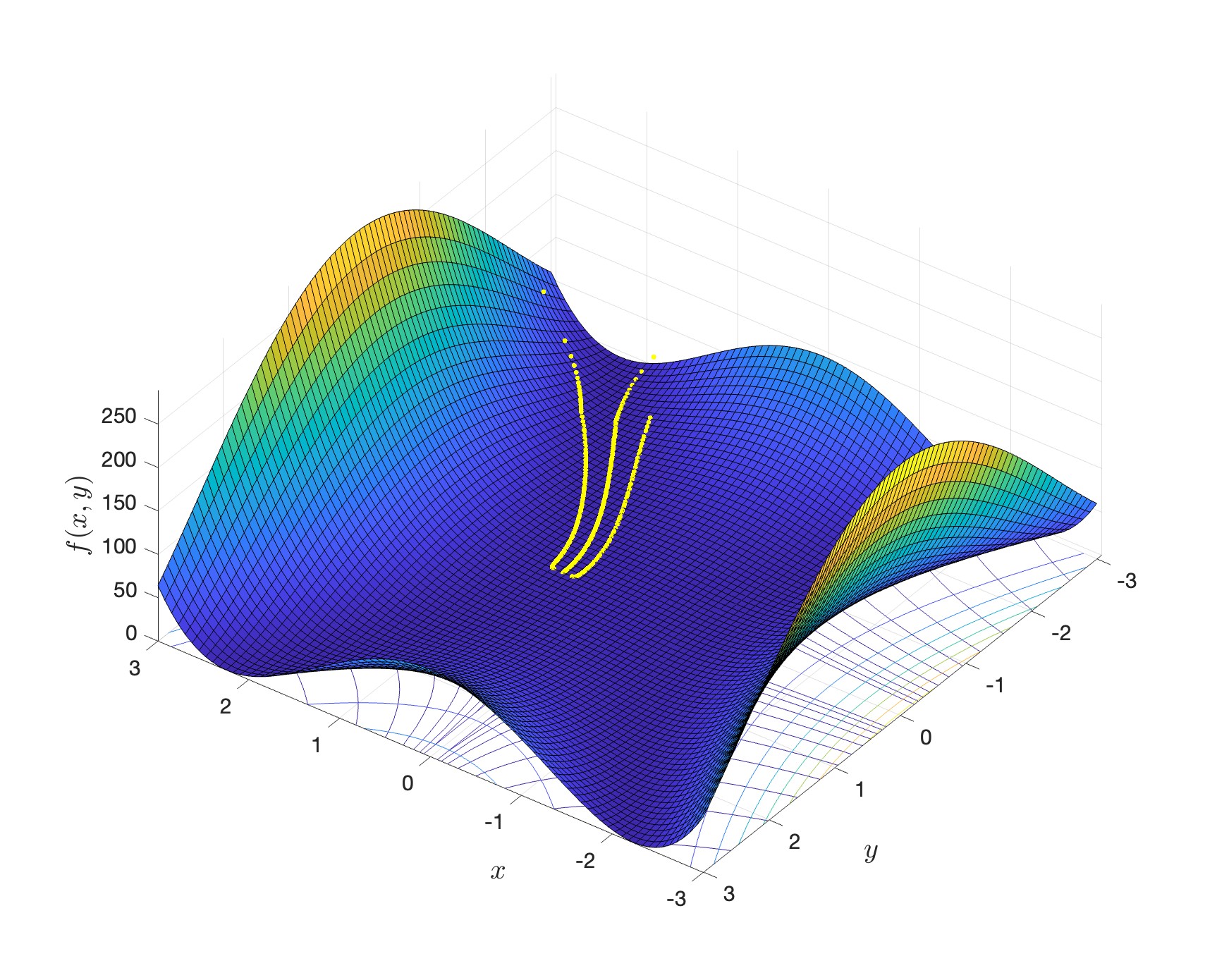}
  \caption{$f(x,y)=(2x^2-y^2-1)\p 2$.}
  \label{fig:hyperbola}
\end{subfigure}
\caption{Normalized gradient descent with step sizes $1/(k+1)\p {1/2}$ and $1/(k+1)\p {1/3}$ resp.}
\end{figure}

The following example captures the function studied in \cite{beneventano2025gradient} as a special case.

\begin{example}
    \label{eg:bilinear}
    Let $f:\R^m\times \R^n\to \R$ be defined by $$f(x,y)=(\langle Ax,y \rangle-1)\p 2$$ where $A\in \mathbb{R}^{m\times n}$. Each flat minimum is d-stable, and the set of flat minima is an almost sure asymptotic 2-attractor. For almost every $A$, each flat minimum is asymptically 2-stable.
\end{example} 
\begin{proof}
    Consider a singular value decomposition $A = U\Sigma V^T$. The change of variables $(V^T  x, U^T y)$ reduces the problem to $f(x,y)=(\langle \Sigma x,y \rangle-1)\p 2$ by \cref{prop:change}. Without loss of generality, we may consider that $\Sigma$ has $n$ nonzero entries and $m=n$. A further change of variables $(x+y,x-y)/\sqrt{2}$ reduces the problem to $f(x,y)=(\langle \Sigma x,x \rangle/2 - \langle \Sigma y,y \rangle/2 -1)\p 2$. We are thus led to consider a multivariate extension of \cref{eg:ellipse}, namely $$f(x)=(a_1x_1^2+\dots+a_n x_n^2-1)\p 2/2$$
    where $a \in \mathbb{R}^{n}$. 

    Compute $$\nabla f(x) = 2(a_1x_1^2+\dots+a_n x_n^2-1)\begin{pmatrix}a_1x_1 \\ \vdots \\ a_n x_n\end{pmatrix}\text{and}~ \widehat\nabla f(x) = \frac{\sgn(a_1x_1^2+\dots+a_n x_n^2-1)}{\sqrt{a_1^2x_1^2+\dots+a_n^2x_n^2}}\begin{pmatrix}a_1x_1 \\ \vdots \\ a_n x_n\end{pmatrix}.$$
    Since $\nabla f(x)=0$ implies $f(x)=0$ or $a_ix_i=0$ for all $i\in \llbracket 1,n\rrbracket$, every local minimum is a global minimum. A minimum of $f$ is flat iff $a_ix_i = 0$ for all $i\notin I =\arg\min \{a_i : a_i > 0\}$ by \cite[Example 5.4]{josz2025flat}. The case where $I$ is empty is trivial. Without loss of generality, assume $\min \{ a_i : a_i > 0\}=1$ and $a_i \neq 0$ for all $i\in \llbracket 1,n\rrbracket$. Thus, a minimum $x$ of $f$ is flat iff $x_{I\p c} =0$ where $I\p c = \llbracket 1,n\rrbracket \setminus I$.  
    
    Due to the expression of $\nabla f$ above, the quantity $C:\R\p n\to\eR$ defined by
 \begin{equation*}
    C_i(x) = 
    \left\{
    \begin{array}{ll}
         |x_i|/|x_I| & \text{if} ~ i\in I~\text{and}~x_I\neq 0, \\
         |x_i|/|x_I|^{a_i} & \text{if} ~ i\notin I~\text{and}~x_I\neq 0, \\
         \infty & \text{if}~x_I=0,
    \end{array}
    \right.
\end{equation*}
 is conserved. Let $c(x)=\|C(x)-C(\overline{x})\|_1$ where $\overline{x}$ is a flat minimum. Note that $C_i(\overline{x}) = 0$ for all $i\notin I$. For all $x\in \R\p n$ near a point in $\dom c$ and all sufficiently small $\alpha>0$, we have 
 $$\forall i \notin I, ~~~ C_i(x\p +) \leq C_i(x) \exp\left( -\frac{a_i(a_i-1)}{4(a_1^2x_1^2+\cdots+a_n^2x_n^2)}\alpha \p 2\right),$$ as in \cref{eg:ellipse}, and 
 $$\forall i \in I, ~~~ C_i(x^+)  = \frac{|x_i-\alpha s a_i x_i|}{|x_I-\alpha s a_I x_I|}
 = \frac{|x_i||1-\alpha s a_i|}{|x_I||1-\alpha s a_i|}
 = \frac{|x_i|}{|x_I|} = C_i(x).$$ 
Thus $c(x^+)\leq c(x)$ and $c$ is d-Lyapunov near $\overline x$. Moreover, $f(x) + c(x) = 0$ implies that $|x_i| = |\overline{x}_i|$ for all $i\in \llbracket 1,n \rrbracket$, so $\overline x$ is a strict local minimum of $f+c$. By \cref{thm:stable_point}, $\overline x$ is d-stable. Note that to prove stability, we could have dropped one of the components of $C$ with index $i\in I$. In particular, if $I$ is a singleton, which happens almost surely, then $c$ is actually 2-d-Lyapunov near any point in $[0<c<\infty]$. In that case, $\overline{x}$ is asymptotically 2-stable by \cref{thm:stable_set}.

Let $g:\R\p n\to\eR$ be defined by $g(x)=\|C_{I^c}(x)\|_1$. Similarly, we have 
$$g(x\p +) \leq g(x) \exp\left( -\frac{\max_{i \notin I} a_i(a_i-1)}{4(a_1^2x_1^2+\cdots+a_n^2x_n^2)}\alpha \p 2\right)$$ and so $g$ is 2-d-Lyapunov near any point in $[f=0]\cap[0<g<\infty]$. Coercivity follows from similar arguments as in \cref{eg:ellipse}. Since $[f=0]\cap[g=0]$ are the flat minima of $f$, attractiveness follows by \cref{thm:attractor}. As in \cref{eg:ellipse}, it is also possible to apply \cref{thm:implicit_NGD} using a $C\p {2,2}$ d-Lyapunov function.
\end{proof}

\begin{example}
    \label{eg:monomial}
    The flat global minima of $f(x) = (x^\upsilon - 1)\p 2$ where $x^\upsilon = x_1^{\upsilon_1}\cdots x_n^{\upsilon_n}, \upsilon \in {\mathbb{N}^*}^n$, namely 
    $$|x_i| = \frac{\sqrt{\upsilon_i}}{\sqrt{\upsilon_1^{\upsilon_1}\cdots\upsilon_n^{\upsilon_n}}^{1/|\upsilon|_1}}, ~~~ i = 1, \hdots, n,$$
    for any choice of signs such that $x^\upsilon=1$, are asymptotic 2-attractors.
\end{example}
\begin{proof}
As shown in \cite[Example 5.6]{josz2025flat}, a conserved quantity is given by
    $$C(x) = \begin{pmatrix} \upsilon_nx_1^2 - \upsilon_1x_n^2\\ \vdots \\ \upsilon_n x_{n-1}^2 - \upsilon_{n-1}x_n^2 \end{pmatrix}$$
    and a global minimum is flat iff $C(x) = 0$.
    Also, for all $x\in\R\p n$ near a nonflat global minimum, we have $\langle C(x),C(\nabla f(x))\rangle \leq -\omega|\nabla f(x)|\p 2$ for some constant $\omega>0$. With $c(x) = |C(x)|\p 2/4$, this yields 
    $$\forall x\in [f=0]\cap[0<c],~\forall u \in \widehat\nabla f(x),~~~ \langle \nabla\p 2 c(x)u,u\rangle = \langle C(x),C(u)\rangle < 0,$$ using \cref{prop:conserved}. Thus $c$ is 2-d-Lyapunov near any nonflat global minimum, according to \cref{prop:sufficient_dL2}.
    
    Let $\overline{x}$ be a flat global minimum. By \cref{prop:projection_composition}, $f=F\p 2$ where $F(x)=x\p \upsilon -1$ satisfies the normalized perturbed projection formula at $\overline x$ along $[f=0]$ (one could also use symmetry here, via \cref{prop:projection_symmetry}). By \cref{prop:strict}, $\overline{x}$ is a strict local minimum of $f+c$ since
\begin{equation*}
        \sp\widehat\nabla f(\overline{x}) + \Im C'(\overline{x})^* = \Im\begin{pmatrix}
            \upsilon_1 \overline{x}^{\upsilon}/\overline{x}_1 &
2\upsilon_n \overline{x}_1 & \cdots & 0 \\ 
            \vdots & \vdots & \ddots & \vdots \\
            \upsilon_{n-1} \overline{x}^{\upsilon}/\overline{x}_{n-1} & 0 & \cdots & 2\upsilon_n \overline{x}_{n-1} \\
            \upsilon_n \overline{x}^{\upsilon}/\overline{x}_n & - 2\upsilon_1 \overline{x}_n & \cdots & - 2\upsilon_{n-1} \overline{x}_n
        \end{pmatrix} = \mathbb{R}^n.
    \end{equation*}
    Indeed, the first column is orthogonal to the others which are clearly linearly independent. Therefore, flat minima are d-stable by \cref{thm:stable_set}. 
    
    But $f+c$ is also coercive: if $|x|\to \infty$ and $f+c$ is bounded, then some variable, say $x_1$, diverges, and so does $x_n$, and thus every other variable, which is impossible since $x\p \upsilon =1$. Flat global minima of $f$, i.e., $\arg\min f+c$, are hence asymptotic 2-attractors by \cref{thm:attractor} (it is also possible to apply \cref{thm:implicit_NGD}). This fact alone implies that the flat global minima are the d-stable global minima, but we wanted to illustrate the different tools available.
\end{proof}

While normalized gradient descent is initialized near the global minimum $(2,2^{-1/2})$ in \cref{fig:monomial_xy2}, it converges to the flat global minimum $(2^{-1/3},2^{1/6})$.

\begin{figure}[H]
\centering
\begin{subfigure}{.49\textwidth}
  \centering
  \includegraphics[width=1\textwidth]{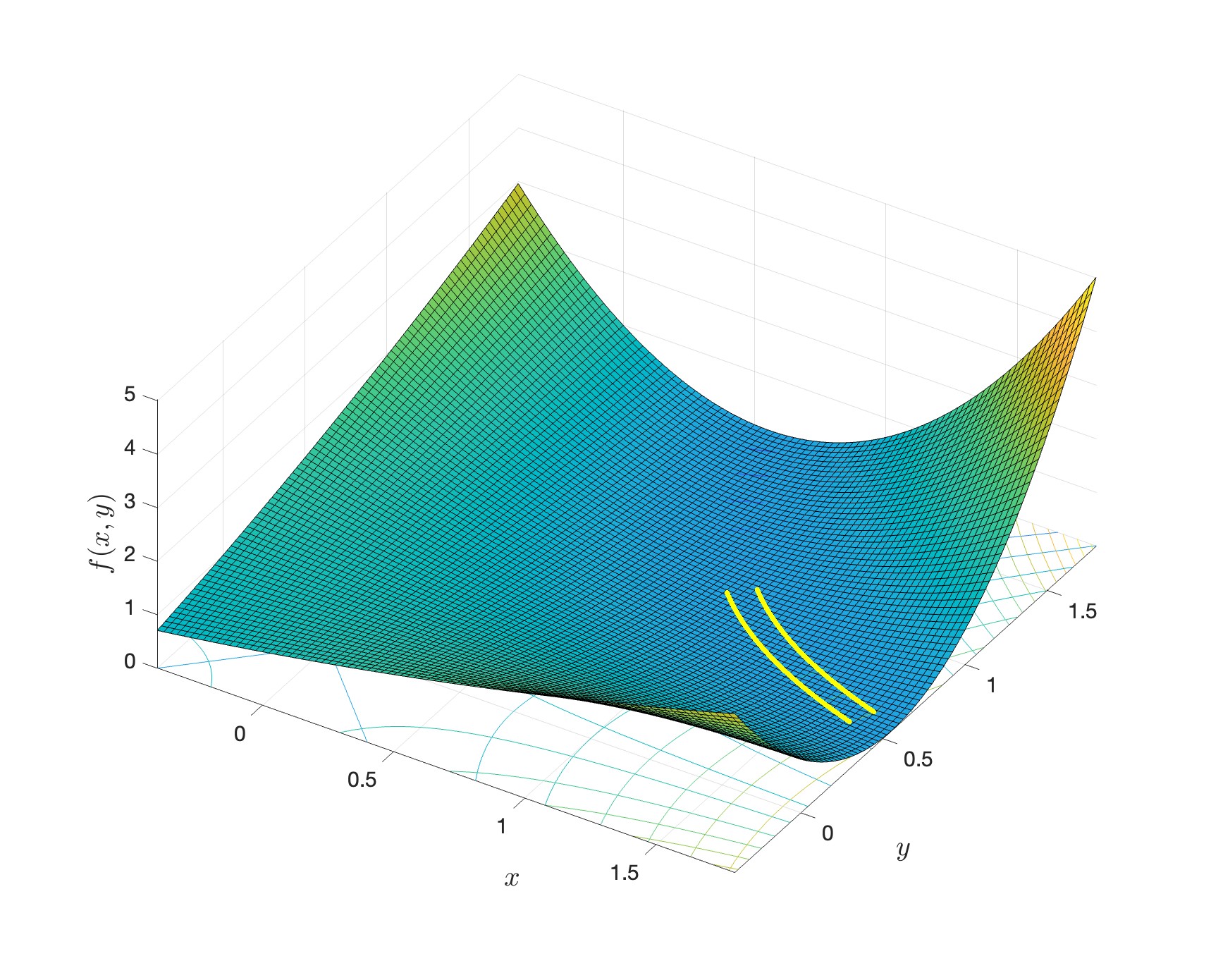}
  \caption{$f(x,y)=(xy-1)\p 2$.}
  \label{fig:monomial_xy}
\end{subfigure}
\begin{subfigure}{.49\textwidth}
\centering
  \includegraphics[width=1\textwidth]{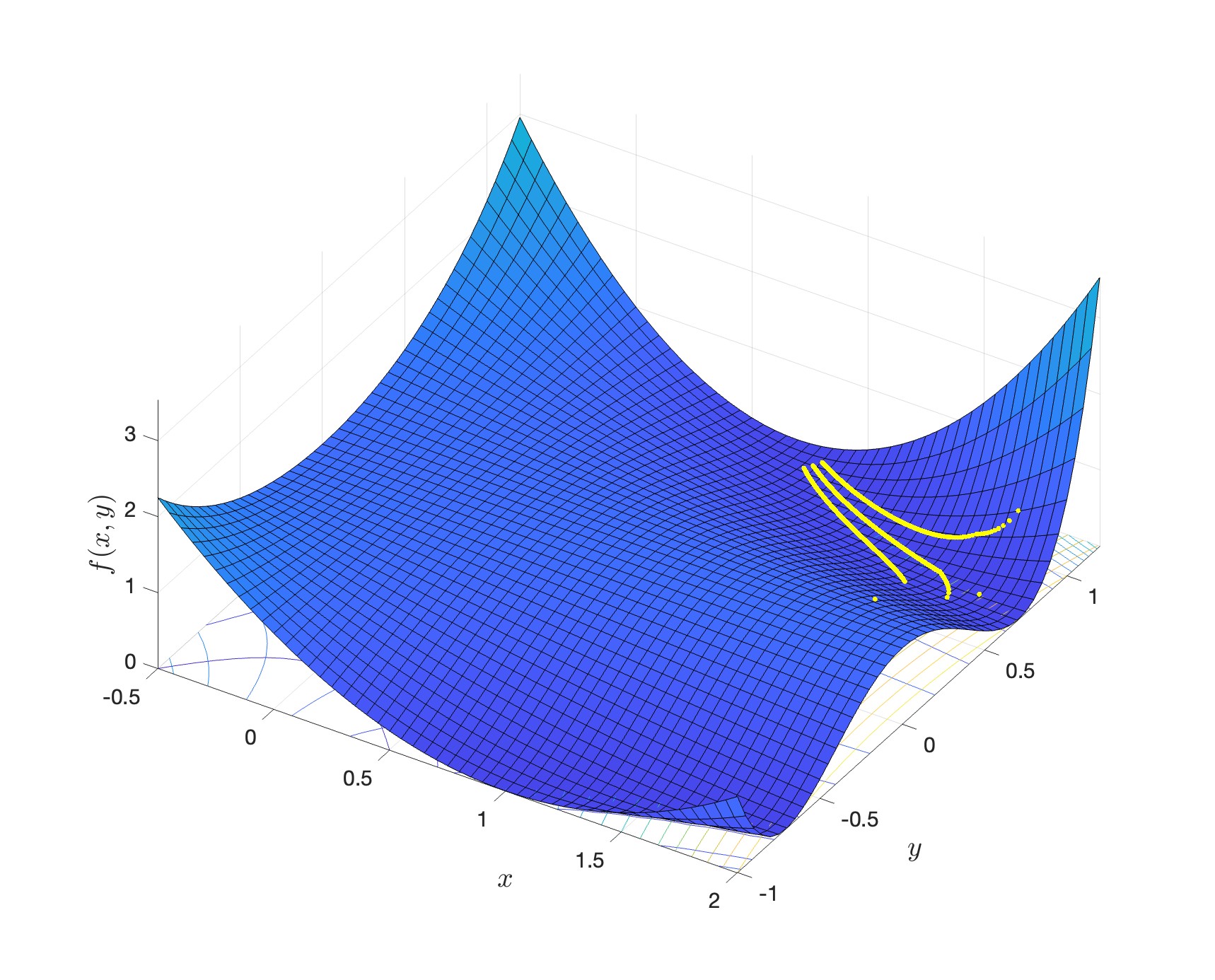}
  \caption{$f(x,y)=(xy^2-1)\p 2$.}
  \label{fig:monomial_xy2}
\end{subfigure}
\caption{Normalized gradient descent with step sizes $0.1$ and $0.5/(k+1)\p {1/3}$ resp.}
\end{figure}

\begin{example}
    \label{eg:mf1_01}
    The flat global minima $\pm (0,2^{1/4},2^{-1/4})$ of $f(x) = |x_1x_3|+|x_2x_3-1|$ are asymptotic 2-attractors.
\end{example}
    \begin{proof} 
    By \cite[Theorems 8.9 and 10.6]{rockafellar2009variational} (see also \cite[Fact 3.25]{josz2025flat}), $f$ is regular,
    \begin{equation*}
       \partial f(x) = \left\{ \begin{pmatrix}
            \lambda_1 x_3 \\
            \lambda_2 x_3 \\
            \lambda_1 x_1 + \lambda_2 x_2
        \end{pmatrix}
        ~,\hspace{3mm} \lambda_1\in\sign(x_1x_3),~ \lambda_2 \in \sign(x_2x_3-1)
        \right\},
    \end{equation*}
    and $\lip f(x) = \sqrt{2x_3\p 2+ (|x_1|+|x_2|)\p 2}$ for all $x\in\R\p 3$. Also,
    $$
        \widehat\nabla f(x) = \left\{ \begin{pmatrix}
            \lambda_1 x_3 \\
            \lambda_2 x_3 \\
            \lambda_1 x_1 + \lambda_2 x_2
        \end{pmatrix}\left/\left|\begin{pmatrix}
            \lambda_1 x_3 \\
            \lambda_2 x_3 \\
            \lambda_1 x_1 + \lambda_2 x_2
        \end{pmatrix}\right|\right.
        ~,\hspace{3mm} \lambda_1,\lambda_2 \in \{-1,1\}
        \right\}
    $$
    for all $x \in \arg\min f = \{(0,t,1/t) : t\neq 0\}$. Accordingly, let $x_t=(0,t,1/t)$. Since $\lip f(x_t) =\sqrt{ 2/t\p2 + t\p 2}$ is strictly minimized at $t = 2\p{1/4}$,  by \cite[Corollary 3.18]{josz2025flat} $\pm \overline{x}$ are the flat minima where $\overline{x} = (0,2^{1/4},2^{-1/4})$. 
    The objective $f$ is invariant under the natural action of the Lie group $G = \{ \mathrm{diag}(t,t,1/t) : t\neq 0 \}$
    whose Lie algebra is $ \fg =\sp \{\mathrm{diag}(1,1,-1) \}$.
    By \cref{prop:conserved}, a conserved quantity is given by $C(x) =  x_1^2+x_2^2 - x_3^2$. Consider the function $c(x)=(C(x)-C(\overline{x}))^2/4$ where $C(\pm\overline{x}) = \sqrt{2}/2$. 
    By \cref{prop:conserved}, 
    $$\forall u \in \widehat\nabla f(x_t),~~~ \langle \nabla\p 2 c(x_t)u,u\rangle  = (C(x_t) - C(\overline{x}))C(u) = (t^2 - 1/t^2 -\sqrt{2}/2)\frac{2/t^2 - t^2}{2/t\p 2+t\p 2} < 0$$
    for all $t \neq \pm 2^{1/4}$, i.e., $c(x_t)>0$. Since $f+c$ is coercive, $\pm \overline{x}$ are asympototic 2-attractors by \cref{thm:implicit_NGD}.
\end{proof}

\begin{example}
    \label{eg:mf1_rank1}
    Let $f:\R^m\times \R^n\to \R$ be defined by
        $$f(x,y) = \|x y^T- u v^T \|_1 $$
    where $u\in \mathbb{R}^{m}$ and $v\in \mathbb{R}^{n}$. If $a^T uv^T b \neq 0$ for all $a \in \{-1,1\}^m$ and $b \in \{-1,1\}^n$, then there exist $0<\underline{t} < \overline{t}$ such that $\{ (ut,v/t) :  \underline{t} \leq |t| \leq \overline{t} \}$
is an asymptotic 2-attractor containing the flat global minima.
\end{example}
\begin{proof}
By \cite[Theorem 10.6]{rockafellar2009variational}, $f$ is regular and we have $$\partial f(x,y) = \left\{\begin{pmatrix}
            \Lambda y\\
            \Lambda^T x\end{pmatrix} : \Lambda \in \sign(x y^T- u v^T) \right\}.$$ Naturally, we have $\arg\min f = \{(ut,v/t) : t\neq 0\}$. The assumption on $u$ and $v$ implies that $\Lambda v \neq 0$ and $\Lambda^T u \neq 0$ for all $\Lambda\in \{-1,1\}^{m\times n}$. If $(x,y)\in\arg\min f$, then
        $$\widehat\nabla f(x,y) \subseteq \left\{\frac{1}{\sqrt{\|\Lambda y\|_F\p 2+\|\Lambda^T x\|_F\p 2}}\begin{pmatrix}
            \Lambda y\\
            \Lambda^T x\end{pmatrix} : \Lambda \in \{-1,1\}^{m\times n} \right\}.$$
The objective is invariant under the action $\R\p *\times \R^m\times \R^n\ni  (t,x,y)\mapsto (ut,v/t) \in \R^m\times \R^n$. By \cref{prop:conserved}, a conserved quantity is given by $C(x,y) =  
|x|^2 -|y|^2$. Consider the function $c(x,y) = C(x,y)\p 2/4$. For any $\Lambda\in\{-1,1\}^{m\times n}$, with $w = (\Lambda v/t,\Lambda^T ut)/\sqrt{|\Lambda v|\p 2/t\p 2+|\Lambda\p T u|\p 2 t\p 2}$ we have 
$$\langle \nabla\p 2 c(ut,v/t)w,w\rangle = C(ut,v/t)C(w) = (|u|^2t^2 -|v|^2/ t^2 ) \frac{|\Lambda v|^2 / t^2 -|\Lambda^T u|^2 t^2}{|\Lambda v|\p 2/t\p 2+|\Lambda\p T u|\p 2 t\p 2} < 0$$
for all $|t|$ sufficiently small or large, or equivalently, for all $c(ut,v/t)$ sufficiently large, say above some $\overline{c}>0$. Since $f+\max\{0,c-\overline{c}\}$ is coercive, the conclusion now follows from \cref{thm:implicit_NGD}.
Note that flat minima exist by \cite[Proposition 4.19]{josz2025flat}. 
\end{proof}

\section{Appendix}
\label{sec:Appendix}

\begin{proof}[Proof of \cref{fact:projection}]
    The first two properties are a local version of the tubular neighborhood theorem \cite[Corollary 5.14]{lee2012smooth}. We repeat Lee's argument here for completeness. The map  $dE_{(\overline{x},0)}$ is invertible for two reasons. First, the restriction of $E$ to the zero section $M_0 = \{(x,0) : x\in M\}$ is the obvious diffeomorphism $M_0\to M$, so $dE_{(\overline{x},0)}$ maps the subspace $T_{(\overline{x},0)} M \subseteq T_{(\overline{x},0)}NM$ isomorphically onto $T_{\overline{x}}M$. Second, the restriction of $E$ to the fiber $N_{\overline{x}}M$ is the affine map $v\mapsto \overline{x}+v$, so $dE_{(\overline{x},0)}$ maps $T_{(\overline{x},0)}(N_{\overline{x}}M)\subseteq T_{(\overline{x},0)}NM$ isomorphically onto $N_{\overline{x}}M$. Since $T_{\overline{x}} \R\p n = T_{\overline{x}} M \oplus N_{\overline{x}} M$, this shows that $dE_{(\overline{x},0)}$ is surjective, and hence bijective by virtue of $\dim NM = n$. By the inverse function theorem \cite[Theorem 4.5]{lee2012smooth}, there are connected neighborhoods $U_0$ of $(\overline{x},0)$ and $V_0$ of $E(\overline{x},0) = \overline{x}$ such that $E|_{U_0}:U_0\to V_0$ is a $C\p 1$ diffeomorphism.
    
    The local slice criterion for embedded submanifolds \cite[Theorem 5.8]{lee2012smooth} implies that $M$ is locally closed. Accordingly, let $V \subseteq V_0$ be a neighborhood of $\overline{x}$ in $\R\p n$ such that $M\cap V$ is closed in $V$. There exists a closed set $A$ in $\R\p n$ such that $M\cap V = A \cap V$. Let $\rho>0$ be such that $\overline{B}_{2\rho}(\overline{x}) \subseteq V$ and fix $x \in B_\rho(\overline{x})$. If $y\notin B_{2\rho}(\overline{x})$, then by the triangular inequality $|x-y|\geq |y-\overline{x}|-|x-\overline{x}| > 2\rho - \rho = \rho$. As a result,
    \begin{align*}
        P_M(x) & = \arg\min \{ |x-y|: y\in M\} = \arg\min \{ |x-y|: y\in M \cap \overline{B}_{2\rho}(\overline{x})\} \\
        & = \arg\min \{ |x-y|: y\in A \cap \overline{B}_{2\rho}(\overline{x})\}.
    \end{align*}
    This guarantees the existence of a projection by \cite[Theorem 1.9]{rockafellar2009variational}. More concisely, $\emptyset \neq P_M(x) \subseteq B_{2\rho}(\overline{x})$. 

    We next show uniqueness. Let $p\in P_M(x)$ and $w \in T_p M$. By \cite[Proposition 5.35]{lee2012smooth}, there is a smooth curve $\gamma:J\to \R\p n$ whose image is contained in $M$ where $J$ is an interval of $\R$ containing $0$, $\gamma(0) = p$, and $\gamma'(0)= w$. The optimality of $p$ implies that $|x-p|\leq |x-\gamma(t)|$ for all $t\in J$, so that $$|x-p|\p 2 \leq |(x-p)-(\gamma(t)-p)|\p 2 \leq |x-p|\p 2 - 2 \langle x-p,\gamma(t)-p\rangle+|\gamma(t)-p|\p 2$$ and $$2 \left\langle x-p,\frac{\gamma(t)-p}{|\gamma(t)-p|}\right\rangle \leq |\gamma(t)-p|.$$
    Passing to the limit yields $2\langle x-p,w\rangle \leq 0$. But since $-w\in T_pM$, we also have $2\langle x-p,-w\rangle \leq 0$, and so $\langle x-p,w\rangle = 0$. As $w$ was arbitrary, $v=x-p \in N_pM$. But then $x = p+v = E(p,v)\in B_\rho(\overline{x})\subseteq V_0$, so the projection $p$ is unique. Precisely, $P_M(x) = \{\pi \circ E|_{U_0}\p{-1}(x)\}$ where $\pi:NM\to M$ is the natural projection. The conclusion now follows by letting $U = E|_{U_0}\p{-1}(B_\rho(\overline{x}))$.
\end{proof}


\bibliographystyle{abbrv}    
\bibliography{references}
\end{document}